\documentclass[a4paper,11pt]{amsart}
\usepackage[utf8]{inputenc}
\usepackage[english]{babel}

\usepackage{amsfonts}
\usepackage{arydshln}
\usepackage{appendix}
\usepackage[english]{babel}
\usepackage{amsmath,amssymb,verbatim,mathrsfs,latexsym,paralist}
\usepackage{graphicx}
\usepackage{epstopdf}
\usepackage{indentfirst}
\usepackage{multirow}
\usepackage{stmaryrd}
\usepackage{amsthm}
\usepackage{longtable}
\allowdisplaybreaks[4]
\usepackage{float}

\newcommand{\ad}{\mathrm{ad}}
\usepackage{color}

\newcommand{\C}{\mathbb{C}}
\newcommand{\R}{\mathbb{R}}

\newcommand{\sll}{\mf{sl}}
\renewcommand{\u}{\mf{u}}
\renewcommand{\v}{\mf{v}}
\newcommand{\s}{\mf{s}}
\newcommand{\w}{\mf{w}}
\renewcommand{\c}{\mf{c}}

\newcommand{\GL}{\mathrm{GL}}
\newcommand{\SL}{\mathrm{SL}}

\newcommand{\mf}{\mathfrak}
\newcommand{\g}{\mf{g}}

\newcommand{\ssl}{\mf{sl}}
\renewcommand{\g}{\mf{g}}

\renewcommand{\u}{\mf{u}}
\renewcommand{\v}{\mf{v}}
\newcommand{\cG}{\mathcal{G}}

\newcommand{\su}{\mf{su}(2,1)}
\newcommand{\SU}{\mathrm{SU}(2,1)}
\newcommand{\suc}{\mf{su}(3)}
\newcommand{\SUc}{\mathrm{SU}(3)}

\usepackage[left=1.12in, right=1.12in, top=1.5in, bottom=1.5in]{geometry}

\newcommand{\diag}{\mathrm{diag}}

\newcommand{\Gal}{\mathrm{Gal}}

\numberwithin{equation}{section}
\newtheorem{theorem}{Theorem}[section]
\newtheorem{proposition}[theorem]{Proposition}

\newtheorem{lemma}[theorem]{Lemma}

\theoremstyle{remark}
 
\theoremstyle{remark}
\newtheorem{rmk}[theorem]{Remark}

\title[Subalgebras of the real forms of $\mathfrak{sl}_3(\mathbb{C})$]{The Subalgebras of the Real Forms of  \(\mathfrak{sl}_3(\mathbb{C})\)}
\date{\today}

\begin{document}


\begin{abstract}
We classify the subalgebras of the real forms the complex linear algebra $\mathfrak{sl}_3(\mathbb{C})$, namely the real special linear algebra $\mathfrak{sl}_3(\mathbb{R})$, the special unitary algebra $\mathfrak{su}(3)$, and the generalized special unitary algebra $\mathfrak{su}(2,1)$. Our approach applies Galois cohomology to the known classification of complex subalgebras of $\mathfrak{sl}_3(\mathbb{C})$.  The subalgebras of $\mathfrak{sl}_3(\mathbb{R})$ were previously classified by Winternitz using different techniques. We recover this classification using our cohomological approach and amend minor inaccuracies. Our work, however, constitutes the first complete classifications of the subalgebras of $\mathfrak{su}(3)$ and $\mathfrak{su}(2,1)$.
In addition to illuminating the internal structure of the real forms of $\mathfrak{sl}_3(\mathbb{C})$, our methodology provides a pathway for future investigations into the subalgebra structure of higher-dimensional cases. In addition, the present work and its extensions offer potential applications in representation theory, applied mathematics, and physics.
\end{abstract}

\author[A. Douglas]{Andrew Douglas$^{1,2}$}
\address[]{$^1$Department of Mathematics, New York City College of Technology, City University of New York, Brooklyn, NY, USA.}
\email{adouglas@citytech.cuny.edu}
\address[]{$^2$Ph.D. Programs in Mathematics and Physics, CUNY Graduate Center, City University of New York, New York, NY, USA.}

\author[W. A. de Graaf]{Willem A. de Graaf$^3$}
\address{$^3$Dipartimento di Matematica, Universit\'a di Trento, Trento, Italy}
\email{willem.degraaf@unitn.it}

\maketitle


\section{Introduction}

We classify the subalgebras of the real forms of the complex special linear algebra $\mathfrak{sl}_3(\mathbb{C})$, namely the real special linear algebra $\mathfrak{sl}_3(\mathbb{R})$, the special unitary algebra $\mathfrak{su}(3)$, and the generalized special unitary algebra $\mathfrak{su}(2,1)$.  Our approach combines Galois cohomology with the known classification of complex subalgebras of $\mathfrak{sl}_3(\mathbb{C})$ \cite{dr16a, dr16b, dr18}.

The subalgebras of $\mathfrak{sl}_3(\mathbb{R})$ were previously classified by Winternitz \cite{win04} using different techniques. We recover this classification using our cohomological approach and amend minor inaccuracies. Our work, however, constitutes the first complete classifications of the subalgebras of $\mathfrak{su}(3)$ and $\mathfrak{su}(2,1)$ (cf. \cite{ddg2}).
In addition to illuminating the internal structure of the real forms of $\mathfrak{sl}_3(\mathbb{C})$, our methodology provides a pathway for future investigations into the subalgebra structure of higher-dimensional cases. In addition, the present work and its extensions offer potential applications in representation theory, applied mathematics, and physics.

Classifications of real subalgebras of real semisimple Lie algebras are considerably less common than those of their complex counterparts. However, substantial work in this area does exist beyond Winternitz's classification of real subalgebras of $\mathfrak{sl}_3(\mathbb{R})$ \cite{win04}; none of which, to the best our knowledge, apply Galois cohomology. For instance, in $1977$, Patera and Winternitz classified the subalgebras of all real Lie algebras of dimension no more than four \cite{pw77}. 
Sugiura classified the Cartan subalgebras of real simple Lie algebras \cite{sug}.  Patera, Winternitz, and Zassenhaus examined maximal abelian subalgebras of real and complex symplectic Lie algebras \cite{pwz}.
Patera, Hussin, and Winternitz classified the parabolic subalgebras of the split real form $G_{2(2)}$ of the complex exceptional Lie algebra $G_2$ (see \cite{janisse}). 
And in his $2023$ doctoral dissertation, Janisse completed the classification of the real subalgebrs of $G_{2(2)}$ \cite{janisse}.

The article is organized as follows. Section 2 overviews  a few applications of subalgebras of semisimple Lie algebras in physics, applied mathematics, and representation theory. 
Section 3 provides background information on Galois cohomology. In Section 4, we present our methodology for using Galois cohomology to classify real subalgebras of real semisimple Lie algebras, based on corresponding classifications over the complex numbers. Sections 5, 6, and 7 apply this procedure to the cases of $\mathfrak{sl}_3(\mathbb{R})$,  $\mathfrak{su}(2,1)$, and  $\mathfrak{su}(3)$, respectively. The appendix summarizes the classification of complex subalgebras of $\mathfrak{sl}_3(\mathbb{C})$ from \cite{dr16a, dr16b, dr18}. It also contains
lists of the stabilizers of the subalgebras in $\SL_3(\C)$.

\section{A few  applications of subalgebras of semismple Lie algebras}

In this section, we  briefly describe a few applications of subalgebras of semisimple Lie algebras in physics, applied mathematics, and representation theory, providing  context for the significance of the results and focus of this paper.

The application of Lie groups and their associated Lie algebras is particularly fruitful in the study of differential equations (refer to \cite{haydon, olver} for details). By utilizing the symmetry group of a differential equation, one can often reduce its complexity, transforming a partial differential equation (PDE) into an ordinary differential equation (ODE) or lowering the order of an ODE.

The origins of such applications can be traced back to Sophus Lie. In his work \cite{sophus} (see also \cite{win04}), Lie demonstrated that the symmetry group of a second-order ODE of the form \( y'' = f(x, y, y') \) is \(\mathrm{SL}_3(\mathbb{F})\) (where \(\mathbb{F} = \mathbb{R}\) or \(\mathbb{C}\)) if and only if the equation is linearizable by a point transformation. Furthermore, Winternitz explains that ``A classification of subgroups of \(\mathrm{SL}_3(\mathbb{F})\) provides a classification of possible ways that the symmetry of a linear, or linearizable ODE, can be broken by imposing additional conditions on solutions" \cite{win04}.

Physical models provide additional applications of Lie algebras and their subalgebras, the Interacting Boson Model being one example. This model provides a theoretical framework in nuclear physics designed to describe the collective properties of nuclei in a unified way~\cite{boson}. The theory uses chains of subalgebras, which need to be explicitly described in applications. A deeper understanding of subalgebra structures may also offer new insights into symmetry breaking phenomena.

Examinations and classifications of subalgebras of real and complex semisimple Lie algebras not only illuminate the internal structure of these algebras, but also serve as powerful tools in representation theory. For instance, such classifications facilitate a systematic analysis of branching rules and induction procedures. In connection with branching rules, they also enable detailed investigations of the representation-theoretic behavior of subalgebras, such as determining which subalgebras are \emph{wide} (a subalgebra of a semisimple Lie algebra is called wide if every irreducible representation of the ambient algebra remains indecomposable upon restriction). At present, wide subalgebras are not fully understood; the best-developed theory pertains to regular subalgebras (cf.~\cite{dr25}).

\section{Preliminaries on Galois cohomology}

Here we deal with Galois cohomology relative to the Galois group $\Gamma = \Gal(\C/\R)$. 
For a general treatment, we refer to the book by Serre (\cite{serre}).

Write $\Gamma = \{1,\gamma\}$. Let $\cG$ be a group on which $\Gamma$ acts by automorphisms. 
Let $X$ be a set on which $\Gamma$ acts. Define the maps $\sigma : \cG\to\cG$, $\sigma : X\to X$
by $\sigma(g) = \gamma\cdot g$, $\sigma(x) = \gamma\cdot x$ for $g\in \cG$, $x\in X$. 
We also suppose that $\cG$ acts 
$\Gamma$-equivariantly on $X$, that is, $\sigma(g) \cdot \sigma(x) = \sigma (g\cdot x)$.

An element $g\in \cG$ is called a {\em cocycle} if $g\sigma(g) = 1$. We remark that since we 
use the Galois group $\Gamma$, this is equivalent to, but not the same as, the usual definition. 
Two cocycles $g_1$, $g_2$ are {\em equivalent} if there is an $h\in \cG$ with $h^{-1} g_1  \sigma(h)=g_2$. The set of equivalence classes is called the first cohomology set of $\cG$ and
denoted $H^1 (\cG,\sigma)$, or also just $H^1 \cG$ if $\sigma$ is clear. For a cocycle $g\in \cG$
we denote its equivalence class in $H^1\cG$ by $[g]$.

By $\cG^\sigma$, $X^\sigma$ we denote the elements of $\cG$ and $X$ respectively that are 
fixed under $\sigma$. Let $x_0\in X^\sigma$ and let $Z_{\cG}(x_0) = \{ g\in \cG \mid 
g\cdot x=x\}$ be its stabilizer in $\cG$. We have a natural map $i_* : H^1 Z_{\cG}(x_0)\to 
H^1 \cG$, mapping the class of a cocycle in $H^1 Z_{\cG}(x_0)$ to its class in $H^1 \cG$.
By definition, its kernel is the set of all classes that are sent to the trivial class, i.e.,
$$\ker i_* = \{ [c]\in H^1 Z_{\cG}(x_0) \mid c \text{ is equivalent to 1 in } \cG\}.$$
We now have the following theorem for whose proof we refer to 
\cite[Section I.5.4, Corollary 1 of Proposition 36]{serre}. 

\begin{theorem}\label{thm:galco}
Let the notation be as above. Let $Y=\cG\cdot x_0$ be the orbit of $x_0$.  The orbits of 
$\cG^\sigma$ in $Y^\sigma$ are in bijection with $\ker i_*$. The bijection goes as follows:
let $[c]\in \ker i_*$, then there is a $g\in \cG$ with $g^{-1} \sigma(g) = c$ and $[c]$ 
corresponds to the $\cG^\sigma$-orbit of $g\cdot x_0$.
\end{theorem}

\subsection{Some facts on Galois cohomology sets of algebraic groups}\label{sec:H1facts}

We consider algebraic groups $G$ with an action by $\Gamma$. For these groups we assume that 
the involution $\sigma$ is a {\em complex conjugation}. This means that whenever $f$ is a 
regular function on $G$ also the function $f^\sigma$ given by $f^\sigma(g) = 
\overline{f(\sigma(g))}$ is regular (cf. \cite[Definition 1.7.4]{gowa}). 

By definition a torus is a connected diagonalizable algebraic group. If $T\subset \GL(n,\C)$ is
a torus then there exists an isomorphism $\chi : (\C^*)^d \to T$. Here we give the Galois
cohomology sets of some small-dimensional tori. For explanations, we refer to 
\cite[Examples 3.1.7]{bgl}, \cite[Theorem 3.6]{bg}. As above we suppose that the tori have
a $\Gamma$-action given by the complex conjugation $\sigma$. 

\begin{lemma}\label{lem:T1}
Let $T$ be a 1-dimensional torus, so that there is an isomorphism $\chi : \C^*\to T$. Then
\begin{enumerate}
\item If $\sigma(\chi(t)) = \chi(\bar t)$ then $H^1 T = \{ [\chi(1)] \}$.
\item If $\sigma(\chi(t))=\chi(\bar t^{-1})$ then $H^1 T = \{ [\chi(1)], [\chi(-1)] \}$.
\end{enumerate}
\end{lemma}

\begin{lemma}\label{lem:T2}
Let $T$ be a 2-dimensional torus with an isomorphism $\chi : (\C^*)^2 \to T$. Suppose that  
$\sigma(\chi(s,t))=\chi(\bar t, \bar s)$ or $\sigma(\chi(s,t))
= \chi(\bar t^{-1}, \bar s^{-1})$. In both cases $H^1 T = \{ [\chi(1,1)] \}$.
\end{lemma}

Now we state Sansuc's lemma (\cite{sansuc}, see also \cite[Proposition 10.1]{bg}), which says
that when computing Galois cohomology we can work modulo the unipotent radical.

\begin{lemma}\label{lem:sansuc}
Let $G$ be a linear algebraic group with conjugation $\sigma$. Let $U\subset G$ be a unipotent 
algebraic normal subgroup, stable under $\sigma$. Let $H=G/U$ and consider the projection 
$p : G\to H$. The induced map $p_* : H^1 G\to H^1 H$ is bijective.
\end{lemma}

The next result is \cite[Proposition 3.3.16]{bgl}.

\begin{proposition}\label{prop:compodd}
Let $H\subset \GL(n,\C)$ be a linear algebraic group with conjugation $\sigma$. Let $H^\circ$
denote the identity component of $H$.  Suppose that $H/H^\circ$ is of cardinality $p^n$ for an
odd prime $p$ and $|H^1(H^\circ,\sigma)|<p$. Then the natural map $H^1(H^\circ,\sigma)
\to H^1(H,\sigma)$ is bijective. 
\end{proposition}

We end this section outlining a procedure for computing the first Galois
cohomology set of a non-connected algebraic group. In the sequel we use this
procedure on a few occasions. We refer to \cite{bg} for more details.

Let $Z\subset \GL(n,\C)$ be a linear algebraic group. Let $Z^\circ$ be the
identity component and $C=Z/Z^\circ$ the component group. Let $\sigma :
Z\to Z$ be a complex conjugation. The restriction of $\sigma$ to $Z^\circ$
is also a conjugation, and we also have an induced conjugation $\sigma$ on
$C$. Since $C$ is a finite group we can compute $H^1 (C,\sigma)$ by brute
force. Let $\pi : Z\to C$ be the projection, which induces a map
$\pi_* : H^1(Z,\sigma)\to H^1(C,\sigma)$, $\pi_*([z]) = [\pi(z)]$.
For $[c]\in H^1(C,\sigma)$ we consider the problem to compute the inverse
image $\pi_*^{-1}([c])$. Doing that for all elements of $H^1(C,\sigma)$
obviously solves our problem.

Suppose that there is a cocycle $z$ in $Z$ such that $\pi(z)=c$ (if there is
no such $z$ then $\pi_*^{-1}([c])$ is empty). Define a new conjugation
$\tau : Z\to Z$ by $\tau(g) = z\sigma(g)z^{-1}$. Let $C^\tau = \{
u\in C \mid \tau(u) = u\}$. We have a right action of $C^\tau$ on
$H^1(Z^\circ,\tau)$ by $[h]\cdot u = [\hat{u}^{-1} h \tau(\hat{u})]$, where
$\hat{u}\in Z$ is such that $\pi(\hat{u})=u$. Let $\{[h_1],\ldots,[h_s]\}$ be
a set of $C^\tau$-orbit representatives. Then $\pi_*^{-1}([c]) =
\{[h_1z],\ldots,[h_sz]\}$ (see  \cite[I.\S 5 Proposition 39(ii)]{serre},
\cite{bg}).

\section{Classifying  subalgebras of the real forms of  $\mathfrak{sl}_3(\mathbb{C})$ with Galois cohomology}

 In this section, we describe our procedure for classifying real subalgebras of a real form of $\mathfrak{sl}_3(\mathbb{C})$, up to conjugation in the corresponding Lie group, using Galois cohomology, given the corresponding classification over $\mathbb{C}$. 
We first establish the background theory, terminology, and notation. 

For ease of notation in this section, we set $G=\mathrm{SL}_3(\mathbb{C})$, and $\mathfrak{g}=\mathfrak{sl}_3(\mathbb{C})$. 
Define a conjugation  $\sigma : G\to G$ for each real form whose subalgebras we're examining: 
\begin{equation}
\begin{cases}
\sigma(g) = \bar{g},& ~\text{when classifying subalgebras of}~ \mathfrak{sl}_3(\mathbb{R}),\\
\sigma(g)  =  \bar{g}^{-t},& ~\text{when classifying subalgebras of}~ \suc,~\text{and}\\
\sigma(g) = N \bar{g}^{-t} N^{-1},& ~\text{when classifying subalgebras of}~ \su,\\
\end{cases}
\end{equation}
where 
\begin{equation}\label{hermitian}
N=\begin{pmatrix} 1&0&0\\0&1&0\\0&0&-1\end{pmatrix}.
\end{equation}
The map $\sigma$ is anti-scalar multiplicative in the case of $\mathfrak{sl}_3(\mathbb{C})$, and anti-multiplicative in the remaining two cases.  In each case, $\sigma$ is an involution.

The group of fixed points $G^\sigma = \{ g\in G \mid
\sigma(g)=g\}$ equals $\mathrm{SL}_3(\R) $, $\SUc$, or $\SU$, respectively. The differential of $\sigma$
(denoted by the same symbol) is the map $\sigma : \g\to\g$, with
\begin{equation}
\begin{cases}
\sigma(x) = \bar{x},& ~\text{when classifying subalgebras of}~ \mathfrak{sl}_3(\mathbb{R}),\\
\sigma(x)  =  - \bar{x}^{t},& ~\text{when classifying subalgebras of}~ \suc,~\text{and}\\
\sigma(x) = -N \bar{x}^t N^{-1},& ~\text{when classifying subalgebras of}~ \su.\\
\end{cases}
\end{equation}

In each case, $\sigma$ is an involution. It is anti-scalar multiplicative in the case of $\mathfrak{sl}_3(\mathbb{R})$, and anti-linear in the remaining two cases.
The algebra of fixed points $\g^\sigma = \{ x\in \g \mid \sigma(x)=x\}$ is equal to $\mathfrak{sl}_3(\mathbb{R})$, $\suc$, or $\su$,  respectively.

The group $G$ acts on the subalgebras of $\g$ by $g\cdot \u=g \u g^{-1}$.
As mentioned earlier, the orbits of this action have been classified in \cite{dr16a, dr16b, dr18}. Our objective is to determine the orbits of $G^\sigma$ in
$\g^\sigma$ using Galois cohomology. We note that a subalgebra $\u\subset \g$
has a basis in $\g^\sigma$ if and only if $\sigma(\u) = \u$. We say that
such subalgebras are {\em real}. For a subalgebra $\u$ of $\g$ we define its
stabilizer as
\begin{equation} \mathcal{Z}_G (\mathfrak{u})= \{ g\in G \mid g\cdot \u = \u\}.\end{equation}

We are now ready to describe our procedure. For each subalgebra representative $\mathfrak{u}$ in the classification of $\g$-subalgebras, we consider its $G$-orbit, denoted $\mathcal{O}$. We then determine if $\mathcal{O}$ contains a real subalgebra, which we refer to as a {\it real point} in $\mathcal{O}$.
If there are no real points, we disregard $\mathcal{O}$.
To identify real points in $\mathcal{O}$, we employ the following two results, the first of which is straightforward.

\begin{lemma}
If the $G$-orbit $\mathcal{O}$ of the $\g$-subalgebra $\mathfrak{u}$ has a real point, then $\sigma(\mathfrak{u}) \in \mathcal{O}$.
\end{lemma}
\begin{proof}
Let $g\cdot \mathfrak{u}$ be a real point. Then, $g \cdot \mathfrak{u} = \sigma(g\cdot \mathfrak{u}) = \sigma(g) \cdot \sigma(\mathfrak{u})$. Hence,
$\sigma(\mathfrak{u}) = (\sigma(g)^{-1} g) \cdot \mathfrak{u}$.
\end{proof}

\begin{theorem}\label{th:1} $\mathrm{[}$\cite{borovoi20}, Proposition 1.4 $\mathrm{]}$
Let  $\mathcal{Z}_G (\mathfrak{u})$ be the stabilizer of  $\mathfrak{u}$ in $G$,  and let $g_0 \in G$  be
such that $\sigma(\mathfrak{u}) = g_0^{-1} \cdot \mathfrak{u}$. Then, 
\begin{equation}
g_1 \cdot \mathfrak{u} ~\text{is real if and only if }~  g_1^{-1} \sigma(g_1)\in \mathcal{Z}_G (\mathfrak{u})g_0. 
\end{equation}
\end{theorem}

In all cases that we deal with we manage to find an element $g_0$ as in the theorem by 
computer aided calculation. We do not give the details of these calculations but just the 
element $g_0$. In all cases $g_0$ is a cocycle, that is, we have $g_0 \sigma(g_0)=1$.
Again by explicit calculation we find an element $g_1\in G$ with $g_1^{-1} \sigma(g_1) = g_0$.
(Note that the existence of such an element is only guaranteed if $g_0$ is equivalent to the
trivial cocycle. However, in all cases of interest in this paper we manage to find this $g_1$.)
Then $g_1\cdot \u$ is a real point in the orbit of $\u$. The stabilizer of 
$g_1\cdot \u$ is $g_1\mathcal{Z}_G(\u)g_1^{-1}$. The stabilizer $\mathcal{Z}_G(\u)$ is given
in the tables of the appendix. In view of Theorem \ref{thm:galco}
we then consider the Galois cohomology set $H^1 (g_1\mathcal{Z}_G(\u)g_1^{-1},\sigma)$ in order to determine the real subalgebras that are $\SL_3(\C)$-conjugate,
but not $G^\sigma$-conjugate, to $\u$.

\section{The real subalgebras of $\mathfrak{sl}_3(\mathbb{R})$.}

The real special linear algebra $\mathfrak{sl}_3(\mathbb{R})$ is the Lie algebra of traceless $3 \times 3$ matrices with real entries.  It is eight-dimensional, with basis 

\begin{equation}\label{sl3basiss}
\arraycolsep=1.5pt\def\arraystretch{1.28}
\begin{array}{llllllllllll}
H_\alpha &=& \begin{pmatrix}
1 & 0 & 0\\
0 & -1 & 0\\
0&0&0
\end{pmatrix}, & H_\beta &=& \begin{pmatrix}
0 & 0 & 0\\
0 & 1 & 0\\
0&0&-1
\end{pmatrix}, & \\
X_\alpha &=& \begin{pmatrix}
0& 1 & 0\\
0 & 0 & 0\\
0&0&0
\end{pmatrix}, & X_\beta &=& \begin{pmatrix}
0 & 0 & 0\\
0 & 0 & 1\\
0&0&0
\end{pmatrix}, &
X_{\alpha+\beta} &=& \begin{pmatrix}
0 & 0 & -1\\
0 & 0 & 0\\
0&0&0
\end{pmatrix}, \\
Y_\alpha &=& \begin{pmatrix}
0& 0 & 0\\
1 & 0 & 0\\
0&0&0
\end{pmatrix}, & Y_\beta &=& \begin{pmatrix}
0 & 0 & 0\\
0 & 0 & 0\\
0&1&0
\end{pmatrix}, &
Y_{\alpha+\beta} &=& \begin{pmatrix}
0 & 0 & 0\\
0 & 0 & 0\\
-1&0&0
\end{pmatrix}.
\end{array}
\end{equation}
The real special linear algebra $\mathfrak{sl}_3(\mathbb{R})$ is a non-compact real form of the complex special linear algebra $\mathfrak{sl}_3(\mathbb{C})$.
The Lie group corresponding to $\mathfrak{sl}_3(\mathbb{R})$ is the real special linear group
$\mathrm{SL}_3(\mathbb{R})$. It is the Lie group of $3\times 3$ matrices with real entries and determinant one.

Throughout this section we use the conjugations $\sigma : \SL_3(\C)\to \SL_3(\C)$,
$\sigma : \mathfrak{sl}_3(\C)\to \mathfrak{sl}_3(\C)$ given in both cases by 
$\sigma(u) = \overline{u}$ (complex conjugation of the matrix entries). 

In Theorems \ref{onedsl3t} through \ref{levisl3t}, we classify the real subalgebras of $\mathfrak{sl}_3(\mathbb{R})$, up to conjugation in $\mathrm{SL}_3(\mathbb{R})$.
The classification is presented and summarized in Tables \ref{onedsl3}  through \ref{levisl3}. For the tables in this section,  subalgebras separated by a horizontal dashed line are conjugate in $\mathrm{SL}_3(\mathbb{C})$ (but not under conjugation in $\mathrm{SL}_3(\mathbb{R})$). Further,  $\u \sim \mathfrak{v}$   indicates that subalgebras  $\u$ and  $\mathfrak{v}$ are conjugate with respect to $\mathrm{SL}_3(\mathbb{R})$.
Conjugacy conditions for cases with parameters are established by direct computational analysis.

The tables are organized according to dimension and structure. They are structurally organized into  solvable, semisimple, and Levi decomposable subalgebras.

Note that Tables \ref{onedsl3} through \ref{levisl3} contain both the classification of real subalgebras of $\mathfrak{sl}_3(\mathbb{R})$ from the present article, and also, for comparison, the classification from \cite{win04}. 
There are a few minor oversights in \cite{win04}, which we note in the tables. The tables use the basis of $\mathfrak{sl}_3(\mathbb{R})$ and notation used in \cite{win04}. The basis of $\mathfrak{sl}_3(\mathbb{R})$ used in \cite{win04} is given by
\begin{small}
\begin{equation}\label{winbasis}
\arraycolsep=1.1pt\def\arraystretch{1.4}
\begin{array}{llllllllllll}
K_1 &= \frac{1}{2}&\begin{pmatrix}1 & 0 & 0 \\
0 &-1 & 0 \\
0 & 0 & 0
\end{pmatrix}, & K_2 &= \frac{1}{2}&\begin{pmatrix}
0 & 1 & 0 \\
1 & 0& 0 \\
0 & 0 & 0
\end{pmatrix}, & L_3 &= \frac{1}{2 }&\begin{pmatrix}0 & -1 & 0 \\
1 &0 & 0 \\
0 & 0 & 0
\end{pmatrix}, &
D &=&  \begin{pmatrix}1 & 0 & 0 \\
0 &1 &0 \\
0 & 0 & -2
\end{pmatrix}, \\
P_1 &=& \begin{pmatrix}0 & 0 & 1 \\
0 &0 & 0 \\
0 & 0 & 0
\end{pmatrix}, & 
P_2 &=& \begin{pmatrix}0 & 0 & 0 \\
0 &0 & 1 \\
0 & 0 & 0
\end{pmatrix}, & R_1 &=& \begin{pmatrix}0 & 0 & 0 \\
0 &0 & 0 \\
1 & 0 & 0
\end{pmatrix}, &
R_2 &=& \begin{pmatrix}0 & 0 & 0 \\
0&0 & 0 \\
0 & 1 & 0
\end{pmatrix}.
\end{array}
\end{equation}
\end{small}

\begin{theorem}\label{onedsl3t}
A complete list of inequivalent, one-dimensional real subalgebras of $\mathfrak{sl}_3(\mathbb{R})$, up to conjugation in $\mathrm{SL}_3(\mathbb{R})$, is given in Table
\ref{onedsl3}.
\end{theorem}
\begin{proof}
We first consider all  one-dimensional  subalgebras of $\mathfrak{sl}_3(\mathbb{C})$ from Table \ref{table1sl3C}  that are
presented with a real basis; namely, all subalgebras except $\u^a=\langle H_\alpha + a H_\beta \rangle$ for $a$ not real. 
For all such subalgebras $\u$ presented with a real basis, except $\langle H_\alpha + a H_\beta \rangle$ with $a=0$ or $a=1$ Table \ref{tab:stab1} shows that the
stabilizer is connected, and of the form $T$, $T\ltimes U$ (where $T$ is a subtorus of the
diagonal torus of $\SL_3(\C)$ and $U$ is a normal unipotent subgroup) or isomorphic to 
$\GL_2(\C)$. In the first two cases we have that $H^1 \mathcal{Z}_G(\u)$ is trivial by 
Lemmas \ref{lem:T1}(1), \ref{lem:sansuc}. For the third case we remark that the Galois cohomology
of $\GL_n(\C)$ is known to be trivial (\cite[III.i, Lemma 1]{serre}). 
It follows that the complex orbit of each of these subalgebras necessarily contains exactly one
real orbit. Hence, each 
is listed as a single representative in Table \ref{onedsl3}, just as it was for the complex case in Table \ref{table1sl3C}.  We now consider $\langle H_\alpha + a H_\beta \rangle$ with $a=0$ or $a=1$.

\vspace{3mm}

\noindent $\underline{ \langle H_\alpha + a H_\beta \rangle, a=0, a=1}$:
Note that $\langle H_\alpha +  H_\beta \rangle$ $\sim$ $\langle H_\alpha  \rangle$. Hence, we consider only $\langle H_\alpha +  H_\beta \rangle$. Let 
\begin{equation}\label{firstc}c=\begin{pmatrix}
    0&0&1\\0&-1&0\\1&0&0
\end{pmatrix}.\end{equation}
Then from Table \ref{tab:stab1} we see that the stabilizer of the subalgebra is $T\times \langle
1,c\rangle$, where $T$ is the diagonal torus of $\SL_3(\C)$. The Galois cohomology of
$T$ is trivial and the Galois cohomology of the second factor is $\{[1],[c]\}$. It follows 
that the first Galois cohomology set of the stabilizer has two elements, with nontrivial cocycle represented by $c$. 
We have  $g^{-1} \bar{g} =c$, where 
\begin{equation}\label{firstg}g = \begin{pmatrix}
    1&0&1\\0&\tfrac{1}{2}i&0\\i&0&-i
\end{pmatrix}.\end{equation}
Hence, the subalgebra  spanned by 
\begin{equation}g (H_\alpha+H_\beta) g^{-1}=\begin{pmatrix}
    0&0&-i\\0&0&0\\i&0&0
\end{pmatrix}\end{equation}
 is a second orbit representative. A real basis consists of $i$ times the previous matrix, namely $-X_{\alpha+\beta}+Y_{\alpha+\beta}$.
 We denote this subalgebra by \begin{equation}\u_{1,6}=\langle -X_{\alpha+\beta}+Y_{\alpha+\beta} \rangle\end{equation} in Table \ref{onedsl3}, and we define \begin{equation}\u_{1,5}= \langle H_\alpha+H_\beta \rangle.\end{equation} These two subalgebras are 
 conjugate under $\mathrm{SL}_3(\mathbb{C})$, but not under $\mathrm{SL}_3(\mathbb{R})$.

We now must consider the case where the subalgebra is presented with a non-real basis, namely $\u^a=\langle H_\alpha + a H_\beta \rangle$ for $a\notin \mathbb{R}$.  For this case, we must  first find real points, when they exist.
 
\vspace{3mm}

\noindent  $\underline{ \u^a=\langle H_\alpha + a H_\beta \rangle, a \notin \mathbb{R}}$:
As described in Table \ref{table1sl3C},  $\u^a$ $\sim$ $\u^{b}$ if and only if $a=b, 1-b, \frac{1}{b}, \frac{1}{1-b}, \frac{b}{b-1}$, or $\frac{b-1}{b}$.
Since  $\sigma(\u^a) = \u^{\bar{a}}$, if the orbit of $\u^a$ has a real subalgebra then $\bar{a}=a$, $\bar{a}=1-a$, $\bar{a}=\frac{1}{a}$, 
$\bar{a}=\frac{1}{1-a}$, $\bar{a}=\frac{a}{a-1}$, or $\bar{a}=\frac{a-1}{a}$.
If $\bar{a}=a$, then $a$ itself is real, which has already been considered.  The two cases
 $\bar a =\tfrac{1}{1-a}$ or $\bar a = \tfrac{a-1}{a}$ yield no $a\in\C$, $a\not\in\R$.  
 We consider the remaining cases. 

The condition $\bar a = \tfrac{1}{a}$ is equivalent to $|a|=1$. Set
\begin{equation}g_0 = \begin{pmatrix} 0&0&1\\0&-1&0\\1&0&0\end{pmatrix} \text{ and }
  h=\begin{pmatrix} -\tfrac{1}{2}(1+i) & 0 & \tfrac{1}{2}(-1+i)\\
0&-i&0\\ \tfrac{1}{2}(-1+i) & 0 & -\tfrac{1}{2}(1+i) \end{pmatrix}.
  \end{equation}
Then $g_0^{-1}\cdot (H_\alpha+aH_\beta) = -a(H_\alpha+\bar a H_\beta)$ and $h^{-1}\bar h = g_0$.
The subalgebra $h\cdot \u^a$ is spanned by
\begin{equation}\begin{pmatrix} -\tfrac{1}{2} & 0 & u \\ 0 & 1 & 0 \\-u & 0 & -\tfrac{1}{2}
  \end{pmatrix},\end{equation}
where $u\in \R$, $u\neq 0$. The stabilizer of this algebra is obtained by
conjugating the stabilizer of $\u^a$ by $h$. If $a \neq \tfrac{1\pm
  \sqrt{-3}}{2}$ then it consists of
\begin{equation}T(a,c) = \begin{pmatrix} \tfrac{a+c}{2} & 0 & i\tfrac{a-c}{2} \\
  0 & (ac)^{-1} & 0 \\ i\tfrac{-a+c}{2} & 0 & \tfrac{a+c}{2}\end{pmatrix}.\end{equation}
We have $\overline{T(a,c)} = T(\bar c, \bar a)$. By Lemma \ref{lem:T2} the first
Galois cohomology of the stabilizer is trivial. Hence we get one series of
algebras spanned by the above matrix. If $a=\tfrac{1\pm \sqrt{-3}}{2}$ then
the above group is the identity component and the component group has order 3.
So by Proposition \ref{prop:compodd} the Galois cohomology is also trivial in this case and we
get an instance of the same series of algebras.

Now suppose that $\bar a = 1-a$. This is the same as $a=\tfrac{1}{2}+iy$ with
$y\in \R$. Set
\begin{equation}g_0 = \begin{pmatrix} -1&0&0\\ 0&0&1\\ 0&1&0\end{pmatrix}, \text{ and }
  h=\begin{pmatrix} -i&0&0\\0&  -\tfrac{1}{2}(1+i) & \tfrac{1}{2}(-1+i)\\
  0&\tfrac{1}{2}(-1+i)&-\tfrac{1}{2}(1+i)\end{pmatrix}.\end{equation}
Then $g_0^{-1}\cdot (H_\alpha+aH_\beta) = H_\alpha + \bar a H_\beta$ and $h^{-1}
\bar{h} = g_0$. Furthermore $h\cdot \u^a$ is spanned by 
\begin{equation}\begin{pmatrix} 1&0&0\\ 0&-\tfrac{1}{2}&-y\\ 0&y&-\tfrac{1}{2}
\end{pmatrix}.\end{equation}
We obtain the stabilizer of this algebra as in the previous case and again the
Galois cohomology is trivial. So we get one series of algebras spanned by the
above matrix.

In the last case we have $\bar a = \tfrac{a}{a-1}$ which is equivalent to
$a=x+iy$ with $x^2-2x+y^2=0$. Here we set
\begin{equation}g_0 = \begin{pmatrix} 0&1&0\\1&0&0\\0&0&-1\end{pmatrix}, \text{ and }
  h=\begin{pmatrix} -\tfrac{1}{2}(1+i) & \tfrac{1}{2}(-1+i) & 0 \\
  \tfrac{1}{2}(-1+i) & -\tfrac{1}{2}(1+i) & 0\\0&0&-i \end{pmatrix}.\end{equation}
We have that $g_0^{-1}\cdot \u^a = \u^{\bar a}$ and $h^{-1} \bar{h} = g_0$. The
algebra $h\cdot \u^a$ is spanned by
\begin{equation}\begin{pmatrix} \tfrac{1}{2} & u & 0\\ -u & \tfrac{1}{2} & 0 \\
  0&0&-1\end{pmatrix}.\end{equation}
The Galois cohomology of the stabilizer is trivial, as in the previous two
cases. Hence we again get one series of algebras.

In summary, there are three families of subalgebras in this case:
\begin{equation}\label{threefam}
\begin{array}{lll}
\s_1^u= \big\langle -\frac{1}{2}H_\alpha+\frac{1}{2}H_\beta-u X_{\alpha+\beta}+uY_{\alpha+\beta} \big\rangle, \\[2ex]
\s_2^u=\big\langle H_\alpha+\frac{1}{2}H_\beta-u X_\beta+u Y_\beta\big\rangle,  \\[2ex]
\s^u_3= \big\langle \frac{1}{2}H_\alpha+H_\beta+u X_\alpha-u Y_\alpha\big\rangle, 
\end{array}
\end{equation}
where $u \in \mathbb{R}\setminus \{0\}.$
However, the families are pairwise equivalent. Specifically, 
$G \s_2^u G^{-1} = \s_1^u$, where 
\begin{equation}
    G=\begin{pmatrix}0 & 0 & 1 \\
1&0 & 0 \\
0 & 1 & 0
\end{pmatrix}\in \mathrm{SL}_3(\mathbb{R}).
\end{equation}
And, $K \s_2^u K^{-1} = \s_3^u$, where
\begin{equation}
    K=\begin{pmatrix}0 & 1 & 0 \\
0&0 & 1 \\
1 & 0 & 0
\end{pmatrix}\in \mathrm{SL}_3(\mathbb{R}).
\end{equation}

Hence, we get just one new family in this case, up to equivalence. We choose the second family in Eq. \eqref{threefam} as our representative and define \begin{equation}\u_{1,7}^y= \bigg\langle H_\alpha+\frac{1}{2}H_\beta -yX_\beta+yY_\beta\bigg\rangle, ~ y\in \mathbb{R}\setminus \{0\}.\end{equation}
\end{proof}

\begin{theorem}\label{twodsl3t}
A complete list of inequivalent,  two-dimensional real subalgebras of $\mathfrak{sl}_3(\mathbb{R})$, up to conjugation in $\mathrm{SL}_3(\mathbb{R})$, is given
in Table \ref{twodsl3}.
\end{theorem}
\begin{proof}
In the list of two-dimensional subalgebras of $\mathfrak{sl}_3(\mathbb{C})$ in Table \ref{table2sl3C}, only  the last family of subalgebras, namely $\u^a=\langle X_\alpha, aH_\alpha+(2a+1)H_\beta \rangle$,  is presented with a non-real basis, which occurs for $a\notin \mathbb{R}$.  It is only for this case that we must first find real points, when they exist.

However, from Table \ref{table2sl3C},  $\u^a$ $\sim$ $\u^b$ if and only if $a=b$.
Since $\sigma(\u^a) = \u^{\bar{a}}$, if $\u^a$ has a real subalgebra then $a=\bar{a}$. That is, if the orbit of $\u^a$ (with respect to $\mathrm{SL}_3(\mathbb{C})$) has a real subalgebra, then $a$ is real. Hence, we only need to consider $\u^a$ for $a$ real. 

From Table \ref{tab:stab2} we see that the stabilizer of all subalgebras, except $\langle H_\alpha,H_\beta\rangle$, is of the form $T\ltimes U$ with a component group of order 1 or 3
(where $T$ is a subtorus of the diagonal torus of $\SL_3(\C)$ and $U$ is unipotent) or isomorphic to $\GL(2,\C)$. In all cases we get that the Galois cohomology is trivial by using a combination of Lemmas \ref{lem:T1}, \ref{lem:sansuc} and
Proposition \ref{prop:compodd}, and the fact that the Galois cohomology of
$\GL(n,\C)$ is trivial. 
Hence, in all cases except $\langle H_\alpha,H_\beta\rangle$, each complex orbit  has exactly one real orbit.  Each 
is listed as a single row in Table \ref{twodsl3}, just as it was for the complex case in Table \ref{table2sl3C}.

\vspace{3mm}

\noindent $\u=\underline{\langle H_\alpha,H_\beta\rangle}$: From Table
\ref{tab:stab2} we deduce that $\mathcal{Z}_G(\u) = T\times F$, where
$T$ is the diagonal torus of $\SL_3(\C)$ and $F$ is a group of order 6,
isomorphic to the symmetric group on 3 points. The elements of $F$ can be chosen
to be all real. So the classes in $H^1 F$ coincide with the conjugacy classes in $F$ of
elements of order dividing 2. We obtain that $H^1 F = \{ [1], [c]\}$ with $c$ from \eqref{firstc}.
The Galois cohomology of $T$ is trivial, so $H^1 \mathcal{Z}_G(\u) =
\{[1],[c]\}$. Let $g$ be as in \eqref{firstc}, then $g^{-1}\bar g = c$.
Hence, the subalgebra  $g \langle H_\alpha,H_\beta\rangle g^{-1}$
is the representative of a second orbit; it is spanned by 
\begin{equation}-X_{\alpha+\beta}+Y_{\alpha+\beta}=\begin{pmatrix} 0&0&1\\0&0&0\\-1&0&0 \end{pmatrix} ~\text{and}~ H_\alpha-H_\beta=\begin{pmatrix} 1&0&0\\0&-2&0\\0&0&1\end{pmatrix}.\end{equation}
We denote \begin{equation}\u_{2,5}=\langle H_\alpha, H_\beta \rangle~\text{and}~ \u_{2,6}=\langle -X_{\alpha+\beta}+Y_{\alpha+\beta}, H_\alpha-H_\beta\rangle.\end{equation}
\end{proof}

\begin{theorem}\label{semisimplesl3t}
A complete list of inequivalent, semisimple  real subalgebras of  $\mathfrak{sl}_3(\mathbb{R})$, up to conjugation in $\mathrm{SL}_3(\mathbb{R})$, is given in Table 
\ref{semisimplesl3}.
\end{theorem}
\begin{proof}
Each semisimple subalgebra of $\mathfrak{sl}_3(\mathbb{C})$ in  Table \ref{table5sl3C}  is presented with a real basis. Hence, the first step of finding real points is not needed.

From Table \ref{tab:stab4} we see that the stabilizer of $\langle X_{\alpha+\beta}, Y_{\alpha+\beta}, H_\alpha+H_\beta \rangle$ is isomorphic to $\GL(2,\C)$; hence its Galois
cohomology is trivial. 

Let $\u=\langle X_\alpha+X_\beta, 2Y_\alpha+2Y_\beta, 2H_\alpha+2H_\beta\rangle$. 
Table \ref{tab:stab4} shows that the stabilizer $\mathcal{Z}_G (\mathfrak{u})$ 
is isomorphic to $\mathrm{PSL}(2,\mathbb{C})\times \mu_3$ where $\mu_3$ is the group of order 3, consisting of diagonal matrices $\mathrm{diag}(\omega,\omega,\omega)$ with $\omega^3=1$.
It is known that the first Galois cohomology set of $\mathrm{PSL}(2,\mathbb{C})$ has two elements. By Theorem \ref{prop:compodd} this is also 
the first Galois cohomology set of $\mathcal{Z}_G(\u)$. The nontrivial class is represented by 
the same element $c$ from Eq. \eqref{firstc}. Taking $g \in \mathrm{SL}_3(\mathbb{C})$  from Eq. \eqref{firstg} such that $g^{-1} \bar{g} =c$.
we get a 
representative of the second orbit spanned by 
\begin{equation}
\arraycolsep=1.5pt\def\arraystretch{1.28}
\begin{array}{llllllllllll}
-X_{\alpha+\beta}+Y_{\alpha+\beta}&=&\begin{pmatrix} 0&0&1\\0&0&0\\-1&0&0\end{pmatrix}, & 2X_\alpha-\frac{1}{4}Y_\alpha&=&\begin{pmatrix}
    0&2&0\\-\tfrac{1}{4}&0&0\\0&0&0
\end{pmatrix}, ~\text{and}\\
-\frac{1}{4}X_\beta+2Y_\beta&=&\begin{pmatrix}
    0&0&0\\0&0&-\tfrac{1}{4}\\0&2&0
\end{pmatrix}.
\end{array}
\end{equation}
We define \begin{equation}\u_{3,18}=\bigg\langle -X_{\alpha+\beta}+Y_{\alpha+\beta}, 2X_\alpha-\frac{1}{4}Y_\alpha, -\frac{1}{4}X_\beta+2Y_\beta\bigg\rangle.\end{equation}
\end{proof}

\begin{theorem}\label{threedsl3solvt}
A complete list of inequivalent, three-dimensional solvable  real subalgebras of  $\mathfrak{sl}_3(\mathbb{R})$, up to conjugation in $\mathrm{SL}_3(\mathbb{R})$, is given in Table 
\ref{threedsl3solv}.
\end{theorem}
\begin{proof}
We first consider the three-dimensional solvable subalgebras of $\mathfrak{sl}_3(\mathbb{C})$ in Table \ref{table3sl3C} that are presented with a real basis, namely, each case except
$\langle X_\alpha, X_{\alpha+\beta}, (a-1) H_\alpha+aH_\beta \rangle$ and $\langle X_\alpha, Y_{\beta}, H_\alpha+aH_\beta \rangle$
when $a$ is not real.

Let $\u$ be such a subalgebra. Then, if $\u$ is not one of  $\langle X_\alpha, Y_\beta, H_\alpha+H_\beta \rangle$,
$\langle X_\alpha, X_{\alpha+\beta}, H_\beta \rangle$ then by considering the stabilizer from Table \ref{tab:stab3} we get 
that $H^1 \mathcal{Z}_G(\u)$ is trivial by a combination of Lemmas \ref{lem:T1}, \ref{lem:sansuc} and
Proposition \ref{prop:compodd}, and the fact that the Galois cohomology of
$\GL(n,\C)$ is trivial. So these subalgebras yield a single real subalgebra, which is listed as a 
single representative in Table \ref{threedsl3solv}, just as it was for the complex case in Table \ref{table3sl3C}.

\noindent  $\underline{\langle X_\alpha, Y_\beta, H_\alpha+H_\beta \rangle}$: In this case Table \ref{tab:stab3} shows that $\mathcal{Z}_G(\u) = 
\mathcal{Z}_G(\u)^\circ \times F$ where $F=\{1,c\}$ with $c$ as in \eqref{firstc}. The first Galois cohomology set of $\mathcal{Z}_G(\u)^\circ$ is trivial,
so we see that $H^1 \mathcal{Z}_G(\u)=\{[1],[c]\}$. Let $g \in \mathrm{SL}_3(\mathbb{C})$ be the element from Eq. \eqref{firstg} such that $g^{-1} \bar{g} =c$.
Hence, the subalgebra  $g \cdot \u$ represents a second orbit. We have that $g\cdot \u$ is spanned by $X_\alpha, Y_\beta, -X_{\alpha+\beta}+Y_{\alpha+\beta}$. 
We denote \begin{equation} \u_{3,13}=\langle X_\alpha, Y_\beta, H_\alpha+H_\beta \rangle~ \text{and}~ 
\u_{3,14}=\langle X_\alpha, Y_\beta, -X_{\alpha+\beta}+Y_{\alpha+\beta} \rangle,\end{equation} which are
 conjugate under $\mathrm{SL}_3(\mathbb{C})$, but not under $\mathrm{SL}_3(\mathbb{R})$.

\noindent  $\underline{\langle X_\alpha, X_{\alpha+\beta}, H_\beta \rangle}$: In the same way as in the previous case we see that the Galois cohomology set again has two
elements, with nontrivial cocycle representative $c$ and element $g \in \mathrm{SL}_3(\mathbb{C})$ given by 
\begin{equation}\label{secondcg} c = \begin{pmatrix} -1&0&0\\0&0&1\\0&1&0\end{pmatrix} ~\text{and}~  g = \begin{pmatrix} \tfrac{1}{2}i&0&0\\
0&1&1\\0&i&-i\end{pmatrix},\end{equation}
we have that $g^{-1} \bar{g} =c$.
The subalgebra $g \langle X_\alpha, X_{\alpha+\beta}, H_\beta\rangle g^{-1}$ is spanned by 
$X_\alpha, X_{\alpha+\beta}, X_{\beta}-Y_{\beta}$.
We denote \begin{equation}\u_{3,11}=\langle X_\alpha, X_{\alpha+\beta}, H_\beta \rangle ~\text{and}~ 
\u_{3,12}=\langle X_\alpha, X_{\alpha+\beta}, X_{\beta}-Y_{\beta} \rangle.\end{equation} Again, these are
 conjugate under $\mathrm{SL}_3(\mathbb{C})$, but not under $\mathrm{SL}_3(\mathbb{R})$.

We now focus on the subalgebras  of $\mathfrak{sl}_3(\mathbb{C})$ in Table \ref{table3sl3C} that are presented  with a non-real basis. For these cases, we first must determine if they have a real point.

\noindent  $\underline{\langle X_\alpha, X_{\alpha+\beta}, (a-1) H_\alpha+aH_\beta \rangle, a \notin \mathbb{R}}$:   Let $\u^a$ denote the subalgebra spanned by $X_\alpha$, $X_{\alpha+\beta}$, $H^a=(a-1)H_\alpha+aH_\beta$. 
In \cite{dr16a} it is shown that $\u^a$ and $\u^b$ are $G$-conjugate if and
only if $a=b$ or $a=b^{-1}$. Now $\sigma(\u^a) = \u^{\bar a}$. Hence if the
orbit of $\u^a$ has a real subalgebra then $a=\bar a$ or $\bar a= a^{-1}$.
In the first case $\u^a$ itself is real. So we consider the second case.
In \cite{dr16a} it is shown that $g_0\cdot \u^a = \u^{a^{-1}}$ where
\begin{equation}g_0 = \begin{pmatrix} 1&0&0\\ 0&0&1\\ 0&-1&0\end{pmatrix}.\end{equation}
The stabilizer of $\u^a$ is listed in Table \ref{tab:stab3}. With some explicit computation involving
that stabilizer and the element $g_0$ we see that 
\begin{equation}\hat h=\begin{pmatrix} 1&0&0\\0&0&i\\0&i&0\end{pmatrix}\end{equation}
lies in $\mathcal{Z}_G(\u^a)g_0^{-1} \cap Z^1 G$. Setting
\begin{equation}g_1 = \begin{pmatrix} 1&0&0\\ 0&-i&1\\ 0& -\tfrac{1}{2} & \tfrac{1}{2}i
\end{pmatrix},\end{equation}
we get $g_1^{-1} \sigma(g_1) = \hat h$. We have that $g_1\cdot X_\alpha =
\tfrac{1}{2} i X_\alpha -X_{\alpha+\beta}$, $g_1\cdot X_{\alpha+\beta} = \tfrac{1}{2} X_{\alpha} -iX_{\alpha+\beta}$. Furthermore,
since $\bar a = a^{-1}$ we have $a=x+iy$ with $x^2+y^2=1$. Then
$\tfrac{a+1}{a-1} = i\tfrac{y}{x-1}$. (Note that we cannot have $x=1$ as
$a\neq 1$.) A short calculation shows that $g_1H^a g_1^{-1} = (a-1) v_t$ with
\begin{equation}v_t = \begin{pmatrix} 1 &0&0\\ 0&-\tfrac{1}{2} & -t\\
  0&t & -\tfrac{1}{2} \end{pmatrix}, \text{where} ~
t=\tfrac{y}{2(x-1)}.\end{equation}
Hence $g_1\cdot \u^a$ has real basis $X_\alpha,X_{\alpha+\beta},v_t$. We note that $t$ runs through
$(-\infty,\infty)$ as $x$ runs through $(-1,1)$.  We denote $\u_{3, 8}^t=\big\langle X_\alpha, X_{\alpha+\beta}, H_\alpha +\frac{1}{2}H_\beta- t X_{\beta}+t Y_{\beta}\big\rangle$,
$t\in \mathbb{R}\setminus \{0\}$.

Write $\v = g_1\cdot \u^a$. 
For $g\in\mathcal{Z}_G(\u^a)$ define
$\tau(g) = \hat h \sigma(g) \hat h^{-1}$. We have an isomorphism
$\psi : \mathcal{Z}_G(\u^a)\to \mathcal{Z}_G(\v)$ with $\psi(g) = g_1 gg_1^{-1}$. Then $\psi(\tau(g)) =
\sigma(\psi(g))$. Since $\v$ is real the group $\mathcal{Z}_G(\v)$ is stable under $\sigma$.
This implies that $\tau(\mathcal{Z}_G(\u^a))=\mathcal{Z}_G(\u^a)$. Let $H^1 (\mathcal{Z}_G(\u^a),\tau)$ be the first Galois cohomology
set, where the conjugation acts by $\tau$. Similarly we consider
$H^1(\mathcal{Z}_G(\v),\sigma)$. It follows that $\psi : H^1(\mathcal{Z}_G(\u^a),\tau)\to H^1(\mathcal{Z}_G(\v),\sigma)$,
$\psi([g]) =[\psi(g)]$ is a bijection. The reductive part of $\mathcal{Z}_G(\u^a)$ consists of
$\diag(g_{11},g_{22},g_{33})$. The image under $\tau$ of such an element is
$\diag(\bar g_{11},\bar g_{33},\bar g_{22})$. By Lemmas \ref{lem:T2}, \ref{lem:sansuc} 
it follows that $H^1(\mathcal{Z}_G(\u^a),\tau)$ is
trivial. Hence $H^1(\mathcal{Z}_G(\v),\sigma)$ is trivial as well. We conclude that the
complex orbit of $\v$ has one real orbit, with representative $\u_{3, 8}^t$, defined above.

\noindent  $\underline{\langle X_\alpha, Y_{\beta}, H_\alpha+aH_\beta \rangle, a\notin \mathbb{R}}$:  In this case we let $\u^a$ denote the subalgebra spanned by $X_\alpha, Y_\beta$, $H^a=H_\alpha+aH_\beta$.   
In \cite{dr16a} it is shown that $\u^a$, $\u^b$ are $G$-conjugate if and
only if $a=b$ or $a=b^{-1}$. So the situation is similar to the previous
example. Again we consider the case where $\bar a= a^{-1}$.
We have $c\cdot \u^a = \u^{a^{-1}}$ where $c$ is the element from \eqref{firstc}.
We have $c^2=1$ so that $c$ itself is a cocyle in $\mathcal{Z}_G(\u^a)c^{-1}$. Furthermore,
$c = g_1^{-1} \sigma(g_1)$ with
\begin{equation}g_1 = \begin{pmatrix} 1&0&1\\ 0&i&0\\ \tfrac{1}{2}i & 0 & -\tfrac{1}{2}i
\end{pmatrix}.\end{equation}
We have that $g_1\cdot X_\alpha = -iX_\alpha+\tfrac{1}{2}Y_\beta$, $g_1\cdot Y_\beta =
-iX_\alpha-\tfrac{1}{2} Y_\beta$ and $g_1\cdot H^a = (a-i)v_t$ with
\begin{equation}v_t = \begin{pmatrix} -\tfrac{1}{2}&  0 & t\\
  0&1&0\\ -\tfrac{t}{4} & 0 & -\tfrac{1}{2}\end{pmatrix}.\end{equation}
Here $t=\tfrac{y}{x-1}$ when $a=x+iy$ with $x^2+y^2=1$. So $\g_1\cdot \u^a$ is
spanned by $X_\alpha$, $Y_\beta$, $v_t$ and is therefore real. Exactly as in the previous
example we have that the $G$-orbit of $\g_1\cdot \u^a$ contains exactly one
$G^\sigma$-orbit.  We define  \begin{equation}\u_{3, 10}^t=\bigg\langle X_\alpha, Y_\beta, -\frac{1}{2}H_\alpha +\frac{1}{2}H_\beta- t X_{\alpha+\beta}+\frac{t}{4} Y_{\alpha+\beta}\bigg\rangle,
~t\in \mathbb{R}\setminus \{0\}.\end{equation}
\end{proof}

\begin{theorem}\label{fourfivedsl3solvt}
A complete list of inequivalent, four- and five-dimensional solvable  real subalgebras of  $\mathfrak{sl}_3(\mathbb{R})$, up to conjugation in $\mathrm{SL}_3(\mathbb{R})$, is given in Table 
\ref{fourfivedsl3solv}.
\end{theorem}
\begin{proof}
In the list of four- and five-dimensional solvable subalgebras of $\mathfrak{sl}_3(\mathbb{C})$ in  Table \ref{table4sl3C} only   $\u^a=\langle X_\alpha, X_\beta, X_{\alpha+\beta}, aH_\alpha + H_\beta \rangle$,  is presented with a non-real basis, which occurs for $a\notin \mathbb{R}$.  It is only for this case that we must first find real points, when they exist.  

However, from Table \ref{table4sl3C},  $\u^a$ $\sim$ $\u^b$ if and only if $a=b$.
Since $\sigma(\u^a) = \u^{\bar{a}}$, if $\u^a$ has a real subalgebra then $a=\bar{a}$. That is, if the orbit of $\u^a$ (with respect to $\mathrm{SL}_3(\mathbb{C})$) has a real subalgebra, then $a$ is real. Hence, we only need to consider $\u^a$, for $a$ real. 

Let $\u$ be a solvable subalgebra of $\SL(3,\C)$ of dimension 4 or 5 as given in  Table \ref{table4sl3C}. Suppose that $\u$ is not one of 
$\langle X_\alpha, X_{\alpha+\beta}, H_\alpha, H_\beta\rangle$, $\langle X_\alpha, Y_{\beta}, H_\alpha, H_\beta\rangle$. The stabilizer of $\u$ is given in 
Tables \ref{tab:stab5}, \ref{tab:stab6}. In all cases, by a combination of Lemmas \ref{lem:T1}, \ref{lem:sansuc} and
Proposition \ref{prop:compodd} we see that the first Galois cohomology set of the stabilizer is trivial. 
Hence, in all cases except these two, each complex orbit  has exactly one real orbit, and the subalgebra is listed as a 
 single row in Table \ref{fourfivedsl3solv}, just as it was for the complex case in Table \ref{table4sl3C}.

\noindent $\underline{ \langle X_\alpha,X_{\alpha+\beta},H_\alpha,H_\beta\rangle }$: From Table \ref{tab:stab5} we see that 
$\mathcal{Z}_G(\u) = \mathcal{Z}_G(\u)^\circ \times F$ where $F$ is a group of two elements. The first Galois cohomology of $\mathcal{Z}_G(\u)^\circ$
is trivial. Hence $H^1 \mathcal{Z}_G(\u)$ has two elements. The nontrivial class is 
represented by 
the element $c$ from Eq. \eqref{secondcg}, and $g \in \mathrm{SL}_3(\mathbb{C})$  from the same equation is such that $g^{-1} \bar{g} =c$.
The  second real orbit has representative $g\cdot \u$, which has
basis $X_\alpha$, $X_{\alpha+\beta}$, $X_\beta-Y_\beta$, $2H_\alpha+H_\beta$.
We define \begin{equation} \u_{4,1}=\langle X_\alpha, X_{\alpha+\beta}, H_\alpha, H_\beta \rangle~ \text{and}~ \u_{4,2}=\langle X_\alpha, X_{\alpha+\beta}, X_\beta-Y_\beta, 2H_\alpha+H_\beta \rangle.\end{equation}

\noindent $\underline{ \langle X_\alpha, Y_{\beta}, H_\alpha, H_\beta\rangle }$:
The Galois cohomology set again has two elements in this case. The nontrivial cocycle has representative is
the same element $c$ from Eq. \eqref{firstc}, and $g \in \mathrm{SL}_3(\mathbb{C})$  from Eq. \eqref{firstg} such that $g^{-1} \bar{g} =c$.
Then, the second real orbit has
representative with basis $X_\alpha$, $Y_\beta$, $-X_{\alpha+\beta}+Y_{\alpha+\beta}$, $H_\alpha-H_\beta$.
We define \begin{equation}\u_{4,3}=\langle X_\alpha, Y_{\beta}, H_\alpha, H_\beta \rangle ~\text{and}~ \u_{4,4}=\langle X_\alpha, Y_{\beta}, -X_{\alpha+\beta}+Y_{\alpha+\beta}, H_\alpha-H_\beta \rangle.\end{equation}
\end{proof}

\begin{theorem}\label{levisl3t}
A complete list of inequivalent, Levi decomposable  real subalgebras of  $\mathfrak{sl}_3(\mathbb{R})$, up to conjugation in $\mathrm{SL}_3(\mathbb{R})$, is given in Table 
\ref{levisl3}.
\end{theorem}
\begin{proof}
Each Levi decomposable  subalgebra of $\mathfrak{sl}_3(\mathbb{C})$ in  Table \ref{table6sl3C} is presented with a real basis. Hence, the first step of finding real points is not needed. Further, from Tables \ref{tab:stab5}, \ref{tab:stab6}, \ref{tab:stab7} we see that the stabilizer of each such subalgebra in $\SL(3,\C)$ is isomorphic
to $\GL(2,\C)$. Hence the Galois cohomology of the stabilizer is trivial in all cases, so for all  subalgebras, each complex orbit has a single real orbit. 
\end{proof}

\begin{table} [H]
\renewcommand{\arraystretch}{1.6} 
\caption{One-Dimensional Real Subalgebras of $\mathfrak{sl}_3(\mathbb{R})$ from Theorem \ref{onedsl3t} and \cite{win04}. } \label{onedsl3}
\centering
\scalebox{0.88}{
\begin{tabular}{|c|c|c|c|}
\hline
Dim. &  Subalgebras of  $\mathfrak{sl}_3(\mathbb{R})$ from Theorem \ref{onedsl3t} & Subalgebras of  $\mathfrak{sl}_3(\mathbb{R})$ from \cite{win04}\\
\hline \hline
1 & $\u_{1,1}=\langle X_\alpha + X_\beta \rangle$  & $W_{1,5}=\langle K_2-L_3+P_2 \rangle$
 \\
\hline
1 & $\u_{1,2}=\langle X_\alpha  \rangle$  & $W_{1,4}=\langle K_2-L_3 \rangle$
 \\
\hline
1 & $\u_{1,3}=\langle X_\alpha + H_\alpha+2H_\beta \rangle$  &  $W_{1,3}=\langle D+K_2-L_3\rangle$
 \\
\hline
1 & $\mathfrak{u}^a_{1,4}=\langle H_\alpha + a H_\beta \rangle$, $a\in \mathbb{R}$, $a \neq 0, 1$.  &  $\mathfrak{u}^{ \Psi(\theta) }_{1,4}\sim W_{1,1}^\theta = \langle 2\cos(\theta)K_1+3\sin(\theta) D\rangle$, where\\
& $\mathfrak{u}^a_{1,4} \sim \mathfrak{u}^b_{1,4}$   iff $b=a, 1-a, \frac{1}{a}, \frac{1}{1-a}, \frac{a}{a-1}$, or $\frac{a-1}{a}$ &  $\theta \in \big(0,\arctan\big(\frac{1}{9}\big)\big]$ and $\Psi(\theta)=  \frac{ \cos(\theta)-3\sin(\theta)}{ \cos(\theta)+3\sin(\theta) }.$   \\
 & \big(Equivalently, restricting $a,b \in \big[\frac{1}{2},1\big)$,    & Note: In \cite{win04}, $\theta \in [0,\pi)$, but,  \\
 & $\u_{1,4}^a \sim \u_{1,4}^b$ iff $a=b$. Further, & this yields equivalent subalgebras. \\
 &for any $c\in \mathbb{R} \setminus \big\{0, 1\big\}$ & All non-equivalent subalgebras are included with \\
 &$\u_{1,4}^c \sim \u_{1,4}^a$ for some $a\in \big[\frac{1}{2},1\big)$.\big)& $\theta \in \big[0, \arctan\big( \frac{1}{9} \big) \big]$.  \\
\hline
1 &  $\u_{1,5}=\langle H_\alpha +  H_\beta \rangle$ &   $W_{1,1}^{0}$\\
\hdashline 
1 &  $\u_{1,6}=\langle -X_{\alpha+\beta}+Y_{\alpha+\beta} \rangle$&  $W_{1,2}^0= \langle L_3 \rangle$
\\
\hline
1 &  $\u_{1,7}^{y}=\big\langle H_\alpha +\frac{1}{2}H_\beta-yX_{\beta}+yY_{\beta} \big\rangle$,  $y \in \mathbb{R} \setminus \{ 0\}$. &  $\u_{1,7}^{\frac{1}{4\lambda}} \sim W_{1,2}^\lambda=\langle L_3+\lambda D\rangle$, $\lambda > 0$. \\
& $\u_{1,7}^{y} \sim \u_{1,7}^{y'}$ iff $y =\pm y'$&  $W_{1,2}^\lambda \sim W_{1,2}^\mu$ iff $\lambda =\pm \mu$
\\
\hline
\end{tabular} }
\end{table}

\begin{table} [H]\renewcommand{\arraystretch}{1.6} \caption{Two-Dimensional   Real Subalgebras of $\mathfrak{sl}_3(\mathbb{R})$ from Theorem \ref{twodsl3t} and \cite{win04}. } \label{twodsl3}
\centering
\scalebox{0.99}{
\begin{tabular}{|c|c|c|clclclc|c|}
\hline
Dim. &  Subalgebras of  $\mathfrak{sl}_3(\mathbb{R})$ from Theorem  \ref{twodsl3t} & Subalgebras of  $\mathfrak{sl}_3(\mathbb{R})$ from \cite{win04} \\
\hline \hline
2 & $\u_{2,1}=\langle X_\alpha + X_\beta, X_{\alpha+\beta} \rangle$  &  $W_{2,6}=\langle K_2-L_3+P_2, P_1\rangle$
 \\
\hline
2 & $\u_{2,2}=\langle X_\alpha, H_\alpha+2H_\beta  \rangle$  & $W_{2,3}=\langle K_2-L_3, D\rangle$
 \\
\hline
2 & $\u_{2,3}=\langle X_\alpha, X_{\alpha+\beta} \rangle$  & $W_{2,4}=\langle K_2-L_3, P_1\rangle$
 \\
\hline
2 & $\u_{2,4}=\langle X_\alpha,  Y_\beta \rangle$   & $W_{2,5}=\langle P_1, P_2\rangle$
 \\
\hline
2 & $\u_{2,5}=\langle H_\alpha,  H_\beta \rangle$   & $W_{2,2}= \langle K_1, D\rangle$
 \\
 \hdashline
2 & $\u_{2,6}=\langle  -X_{\alpha+\beta}+Y_{\alpha+\beta}, H_\alpha - H_\beta \rangle$   & $W_{2,1}= \langle L_3, D\rangle$
 \\
\hline
2 & $\u_{2,7}=\langle X_\alpha+X_\beta,  H_\alpha+H_\beta \rangle$   & $W_{2,10}=\big\langle K_1+\frac{1}{2}D, K_2-L_3+P_2\big\rangle$
 \\
\hline
2 & $\u_{2,8}=\langle X_\alpha, -H_\alpha+H_\beta+3X_{\alpha+\beta}\rangle$   & $W_{2,8}=\big\langle K_1-\frac{1}{6}D+P_1, K_2-L_3\big\rangle$
 \\
\hline
2 & $\u_{2,9}=\langle X_\alpha, -2H_\alpha-H_\beta+3Y_{\beta} \rangle$  & $W_{2,9}=\big\langle K_1+\frac{1}{6}D+P_2, P_1\big\rangle$
 \\
\hline
2 & $\u_{2,10}^a=\langle X_\alpha, aH_\alpha+(2a+1)H_\beta \rangle$, $a\in \mathbb{R}$. & $\u_{2,10}^{-\lambda-\frac{1}{2}}\sim W_{2,7}^{\lambda}=\langle K_1+\lambda D, K_2-L_3\rangle$,
 \\
 &$\u_{2,10}^a \sim \u_{2,10}^b$   iff $a=b$ & $\lambda \in \mathbb{R}$.
 \\
\hline
\end{tabular}}
\end{table}

\begin{table} [H]\renewcommand{\arraystretch}{1.6} \caption{Semisimple  Real Subalgebras of $\mathfrak{sl}_3(\mathbb{R})$ from Theorem \ref{semisimplesl3t} and \cite{win04}.    } \label{semisimplesl3}\centering
\scalebox{0.99}{
\begin{tabular}{|c|c|c|clclclc|c|}
\hline
Dim. &  Subalgebras of  $\mathfrak{sl}_3(\mathbb{R})$ from Theorem \ref{semisimplesl3t} & Subalgebras of  $\mathfrak{sl}_3(\mathbb{R})$ from \cite{win04}\\
\hline \hline
3 & $\u_{3,16}=\langle X_{\alpha+\beta}, Y_{\alpha+\beta}, H_\alpha+H_\beta \rangle$  &  $W_{3,12}=\langle K_1, K_2, L_3\rangle$
 \\
\hline
3 & $\u_{3,17}=\langle X_{\alpha}+X_\beta, 2Y_{\alpha}+2Y_{\beta}, 2H_\alpha+2H_\beta \rangle$  &  $W_{3,13}=\langle L_3, P_1+R_1, P_2+R_2\rangle$ 
\\
\hdashline
3 &  $\u_{3,18}=\langle -X_{\alpha+\beta}+Y_{\alpha+\beta},  2X_\alpha-\frac{1}{4}Y_\alpha,  -\frac{1}{4}X_\beta+2Y_\beta\rangle$&
$W_{3,14}=\langle L_3, P_1-R_1, P_2-R_2\rangle$
 \\
\hline
\end{tabular}}
\end{table}

\begin{table} [H]\renewcommand{\arraystretch}{1.6} \caption{Three-Dimensional Solvable  Real Subalgebras of $\mathfrak{sl}_3(\mathbb{R})$ from Theorem \ref{threedsl3solvt} and \cite{win04}.   } \label{threedsl3solv}\centering
\scalebox{0.81}{
\begin{tabular}{|c|c|c|clclclc|c|}
\hline
Dim. & Subalgebras of  $\mathfrak{sl}_3(\mathbb{R})$ from Theorem \ref{threedsl3solvt} &  Subalgebras of  $\mathfrak{sl}_3(\mathbb{R})$ from \cite{win04}\\
\hline \hline
3 & $\u_{3,1}=\langle X_\alpha, X_{\alpha+\beta}, 2H_\alpha+H_\beta \rangle$  &  $W_{3,5}^{\frac{1}{6}}=\big\langle K_1+\frac{1}{6}D, K_2-L_3, P_1\big\rangle$ 
 \\
\hline
3 & $\u_{3,2}=\langle X_\alpha, Y_\beta, H_\alpha-H_\beta  \rangle$  & $W_{3,6}^{\frac{\pi}{2}}= \langle D, P_1, P_2\rangle$
 \\
\hline
3 & $\u_{3,3}=\langle X_\alpha, X_{\alpha+\beta}, 2H_\alpha+H_\beta+X_\beta \rangle$  &    $W^{\frac{2}{3}}_{3,9}=\big\langle P_2-R_2+\frac{2}{3}D,K_2-L_3, P_1\big\rangle$
 \\
\hline
3 & $\u_{3,4}=\langle Y_\alpha, Y_{\alpha+\beta}, 2H_\alpha+H_\beta+X_\beta  \rangle$   & $W_{3,10}=\langle D+K_2-L_3, P_1, P_2\rangle$
 \\
\hline
3 & $\u_{3,5}=\langle X_\alpha+X_\beta, X_{\alpha+\beta}, H_\alpha+H_\beta \rangle$   & $W_{3,7}=\langle K_1+\frac{1}{2}D, K_2-L_3+P_2, P_1\rangle$
 \\
\hline
3 & $\u_{3,6}=\langle X_\alpha,  H_\alpha, H_\beta \rangle$   &  $W_{3,1}=\langle K_1, K_2-L_3, D\rangle$
 \\
\hline
3 & $\u_{3,7}^a=\langle X_\alpha, X_{\alpha+\beta}, (a-1)H_\alpha+aH_\beta\rangle, a \neq \pm 1, a\in \mathbb{R}$. &  $W^\lambda_{3,5}=\langle K_1+\lambda D, K_2-L_3, P_1\rangle$,  $\lambda \in \big(\frac{-1}{2}, \frac{1}{6} \big) \setminus \{ \frac{-1}{6}\}$.\\
& $\u_{3,7}^a \sim \u_{3,7}^b$ iff $a=b$ or $ab=1$  & $\u_{3,7}^{a}\sim W_{3,5}^{\frac{a}{2a-4}}$, $a\in (-1,1)\setminus \big\{\frac{1}{2}\big\}.$ \\
&(Equivalently, restricting $a, b \in(-1,1)$,  & $\u_{3,7}^{\frac{1}{2}} \sim W_{3,2}= \langle D, P_1, K_2-L_3\rangle$
\\
 &$\u_{3,7}^a \sim \u_{3,7}^b$ iff $a=b$. Further, for any $c\in \mathbb{R} \setminus \{ \pm 1 \}$, &Note: $\u^{\frac{1}{2}}_{3,7}\sim \u^{2}_{3,7} \sim W^{\frac{-1}{6}}_{3,5}$, and  there's a minor  
 \\
 &$\u_{3,7}^c \sim \u_{3,7}^a$  for some $a\in (-1,1)$.)& oversight in \cite{win04} since $W_{3,11} \sim W_{3,5}^{\frac{1}{18}}$ ($\sim \u^{\frac{-1}{4}}_{3,7}$), where  \\
 &&$W_{3,11}= \big\langle K_1+\frac{1}{3}D+P_2, K_2-L_3, P_1\big\rangle$.\\
 \hline
 3 & $\u_{3,8}^t=\big\langle X_\alpha, X_{\alpha+\beta}, H_\alpha+\frac{1}{2}H_\beta-tX_\beta+tY_\beta \big\rangle$, where& $W^\lambda_{3,9}=\langle P_2-R_2+\lambda D, K_2-L_3,P_1\rangle$,   $\lambda\in \big(0, \frac{2}{3}\big)$.\\
 & 
 $t\in \mathbb{R}\setminus \{0\}$. &   $\u_{3,8}^{t} \sim W_{3,9}^{\frac{2}{\sqrt{4t^2+9}}}, t\neq 0.$
 \\
 &   $\u_{3,8}^t\sim \u_{3,8}^s$ iff $t=\pm s$ &  $W^\lambda_{3,9} \sim W^\mu_{3,9}$ iff $\lambda=\pm \mu$ \\
 && Note: In \cite{win04}, $\lambda>0$, but, for  $\lambda > \frac{2}{3}$, \\
 &&  $W^\lambda_{3,9}$ $\sim$  $W^{ \frac{3\lambda^2}{2} -\frac{1}{2}- \frac{\sqrt{9\lambda^4-4\lambda^2}}{2}    }_{3,5}$. Observe that \\ && for  $\lambda >\frac{2}{3}$, 
    $\frac{3\lambda^2}{2} -\frac{1}{2}- \frac{\sqrt{9\lambda^4-4\lambda^2}}{2}  \in \big(\frac{-1}{2}, \frac{1}{6}\big) \setminus \big\{ \frac{-1}{6} \big\}$.\\
\hline
3 & $\u_{3,9}^a=\langle X_\alpha, Y_\beta, H_\alpha+aH_\beta \rangle, a\neq \pm1$   & $W^\theta_{3,6}=\big\langle 2 \cos(\theta) K_1 +\frac{1}{3} \sin(\theta) D, P_1, P_2  \big\rangle, \theta \in \big(0, \frac{\pi}{2}\big)\setminus \big\{  \frac{\pi}{4}\big\}$\\ 
&$\u_{3,9}^a \sim \u_{3,9}^b$ iff $a=b$ or $ab=1$ &$\u_{3,9}^{\frac{3\cos(\theta)-\sin(\theta)}{3\cos(\theta)+\sin(\theta)}} \sim W_{3,6}^\theta$
\\
& (Equivalently, restricting $a, b \in(-1,1)$, & Note: $\u_{3,9}^{\frac{1}{2}}\sim W_{3,3}=\big\langle K_1+\frac{1}{6}D, P_1, P_2\big\rangle = W^{\frac{\pi}{4}}_{3,6}$ \\
&$\u_{3,9}^a \sim \u_{3,9}^b$ iff $a=b$. Further, for any $c\in \mathbb{R} \setminus \{ \pm 1 \}$,& \\
&$\u_{3,9}^c \sim \u_{3,9}^a$  for some $a\in (-1,1)$.)& \\
\hline
3 & $\u_{3,10}^t=\big\langle  X_\alpha, Y_\beta, -\frac{1}{2}H_\alpha +\frac{1}{2}H_\beta-t X_{\alpha+\beta} +\frac{t}{4}Y_{\alpha+\beta} \big\rangle$, & 
$\u_{3,10}^{\frac{1}{2\lambda}} \sim W_{3,8}^\lambda =\langle L_3+\lambda D, P_1, P_2\rangle$, $\lambda> 0$. \\
& where $t\in \mathbb{R} \setminus \{0\}$.  $\u_{3,10}^t \sim \u_{3,10}^s$ iff $t=\pm s$&\\  
\hline
3 & $\u_{3,11}=\langle X_\alpha, X_{\alpha+\beta}, H_\beta \rangle$   &  $W_{3,5}^{\frac{-1}{2}}=\big\langle K_1-\frac{1}{2}D,K_2- L_3, P_1\big\rangle$ \\
 \hdashline
 3 & $\u_{3,12}=\langle X_\alpha, X_{\alpha+\beta},  X_\beta-Y_\beta \rangle$ &  $W^0_{3,9}=\langle P_2-R_2, K_2-L_3, P_1\rangle$
 \\
\hline
3 & $\u_{3,13}=\langle X_\alpha, Y_\beta, H_\alpha+H_\beta \rangle$  & $W_{3,6}^0=\langle K_1, P_1, P_2\rangle$
 \\
 \hdashline
3 & $\u_{3,14}=\langle X_\alpha, Y_\beta, -X_{\alpha+\beta}+Y_{\alpha+\beta} \rangle$  &  $W_{3,8}^0=\langle L_3,P_1, P_2  \rangle$
 \\
\hline
3 & $\u_{3,15}=\langle X_\alpha, X_{\beta}, X_{\alpha+\beta} \rangle$   & $W_{3,4}=\langle K_2-L_3, P_1, P_2\rangle$
 \\
\hline
\end{tabular}}
\end{table}

\begin{table} [H]\renewcommand{\arraystretch}{1.6} \caption{Four- and Five-Dimensional Solvable Real Subalgebras of $\mathfrak{sl}_3(\mathbb{R})$
from Theorem \ref{fourfivedsl3solvt} and \cite{win04}.   } \label{fourfivedsl3solv}\centering
\scalebox{0.9}{
\begin{tabular}{|c|c|c|clclclc|c|}
\hline
Dim. &  Subalgebras of  $\mathfrak{sl}_3(\mathbb{R})$ from Theorem \ref{fourfivedsl3solvt} & Subalgebras of  $\mathfrak{sl}_3(\mathbb{R})$ from \cite{win04} \\
\hline \hline
4 & $\u_{4,1}=\langle X_\alpha, X_{\alpha+\beta}, H_\alpha, H_\beta \rangle$ & $W_{4,2}=\big\langle D, P_1, K_1-\frac{1}{6}D, K_2-L_3 \big\rangle$
 \\
 \hdashline
4 & $\u_{4,2}=\langle X_\alpha, X_{\alpha+\beta}, X_\beta-Y_\beta, 2H_\alpha+H_\beta \rangle$ & $W_{4,5}=\big\langle P_2-R_2, K_1+\frac{1}{6}D, K_2-L_3, P_1 \big\rangle$
\\
\hline
4 & $\u_{4,3}=\langle X_\alpha, Y_\beta, H_\alpha, H_\beta  \rangle$  & $W_{4,3}=\big\langle K_1+\frac{1}{6}D,P_1, K_1-\frac{1}{6}D, P_2 \big\rangle$
 \\
 \hdashline
 4 & $\u_{4,4}=\langle X_\alpha, Y_\beta,  -X_{\alpha+\beta} +Y_{\alpha+\beta}, H_\alpha - H_\beta  \rangle$  &
 $W_{4,4}=\langle L_3, D, P_1, P_2 \rangle$
 \\
\hline
4 & $\u_{4,5}=\langle X_\alpha, X_\beta, X_{\alpha+\beta}, H_\alpha+H_\beta \rangle$  &  $W_{4,6}^{\arctan(3)}=\langle 2K_1+D, K_2-L_3, P_1, P_2\rangle $  
 \\
\hline
4 & $\u_{4,6}^a=\langle X_\alpha, X_\beta, X_{\alpha+\beta}, aH_\alpha+H_\beta  \rangle, a\neq \pm 1$, $a\in \mathbb{R}$.   & 
$W_{4,6}^{\theta}= \big\langle  2 \cos(\theta) K_1+\frac{1}{3} \sin(\theta)D, K_2-L_3, P_1, P_2  \big\rangle$, 
 \\
 &$\u_{4,6}^a \sim \u_{4,6}^b$  iff $a=b$& $\theta \in  (0, \pi) \setminus \big\{   \frac{3\pi}{4}, \arctan(3)\big\}.$
 \\
 & & $\u_{4,6}^{\frac{3\cos(\theta)+\sin(\theta)}{2 \sin(\theta)}} \sim W_{4,6}^\theta$
 \\
\hline
4 & $\u_{4,7}=\langle X_\alpha,  X_\beta, X_{\alpha+\beta}, H_\alpha \rangle$   &   $W_{4,6}^{0} \sim W_{4,6}^\pi$ \\
\hline
4 & $\u_{4,8}=\langle X_\alpha, X_{\beta}, X_{\alpha+\beta}, H_\alpha-H_\beta\rangle$   & $W_{4,6}^{\frac{3\pi}{4}}=\big\langle 2 K_1-\frac{1}{3}D, K_2-L_3, P_1, P_2\big\rangle $ 
 \\
\hline
5 & $\u_{5,1}=\langle X_\alpha, X_\beta, X_{\alpha+\beta}, H_\alpha, H_\beta \rangle$   & $W_{5,3}=\langle K_1, D, K_2-L_3, P_1, P_2\rangle$
 \\
\hline
\end{tabular}}
\end{table}

\begin{table} [H]\renewcommand{\arraystretch}{1.6} \caption{Levi Decomposable  Real Subalgebras of $\mathfrak{sl}_3(\mathbb{R})$ from Theorem \ref{levisl3t} and \cite{win04}.   } \label{levisl3}\centering
\scalebox{0.88}{
\begin{tabular}{|c|c|c|clclclc|c|}
\hline
Dim. &  Subalgebras of  $\mathfrak{sl}_3(\mathbb{R})$  from Theorem \ref{levisl3t} & Subalgebras of  $\mathfrak{sl}_3(\mathbb{R})$ from \cite{win04}\\
\hline \hline
4 & $\u_{4,9}=\langle X_{\alpha+\beta}, Y_{\alpha+\beta}, H_\alpha+H_\beta \rangle \oplus \langle H_\alpha-H_\beta\rangle$  &   $W_{4,1}=\langle L_3, K_1, K_2, D \rangle$
 \\
\hline
5 & $\u_{5,2}=\langle X_{\alpha+\beta}, Y_{\alpha+\beta}, H_\alpha+H_\beta \rangle \inplus \langle X_\alpha, Y_\beta \rangle$  & $W_{5,1}=\langle K_1, K_2, L_3, P_1, P_2\rangle$
 \\
\hline
5 & $\u_{5,3}=\langle X_{\alpha+\beta}, Y_{\alpha+\beta}, H_\alpha+H_\beta \rangle \inplus \langle X_\beta, Y_\alpha \rangle$  & $W_{5,2}=\big\langle K_1-\frac{1}{2}D, P_2, R_2, K_2-L_3, P_1\big\rangle$
 \\
\hline
6 & $\u_{6,1}=\langle X_{\alpha+\beta}, Y_{\alpha+\beta}, H_\alpha+H_\beta \rangle \inplus \langle X_\alpha, Y_\beta, H_\alpha-H_\beta \rangle$  & $W_{6,1}=\langle K_1, K_2, L_3, D, P_1, P_2\rangle$
 \\
\hline
6 & $\u_{6,2}=\langle X_{\alpha+\beta}, Y_{\alpha+\beta}, H_\alpha+H_\beta \rangle \inplus \langle X_\beta, Y_\alpha, H_\alpha-H_\beta \rangle$  &  $W_{6,2}=\big\langle K_1-\frac{1}{2}D, P_2, R_2, K_1+\frac{1}{6}D, K_2-L_3, P_1\big\rangle$
 \\
\hline
\end{tabular}}
\end{table}

\section{The real subalgebras of $\su$}

The generalized special unitary algebra \(\mathfrak{su}(2,1)\) is the real Lie algebra of traceless \(3 \times 3\) complex matrices that preserve the Hermitian form \(N\) from Eq. \eqref{hermitian}.
It is eight-dimensional with basis
\begin{equation}
\arraycolsep=1.5pt\def\arraystretch{1.48}
\begin{array}{llllllllll}
   A_1& = \begin{pmatrix} i&0&0\\0&-i&0\\0&0&0 \end{pmatrix}, 
&A_2 &= \begin{pmatrix} 0&0&0\\0&i&0\\0&0&-i\end{pmatrix}, 
     &A_3 &= \begin{pmatrix} 0&1&0\\-1&0&0\\0&0&0\end{pmatrix},\\[2ex]
      A_4  &= \begin{pmatrix} 0&i&0\\i&0&0\\0&0&0\end{pmatrix},
   &A_5 &= \begin{pmatrix} 0&0&1\\0&0&0\\1&0&0 \end{pmatrix}, 
  &A_6 &= \begin{pmatrix} 0&0&i\\0&0&0\\-i&0&0\end{pmatrix}, \\[2ex]
     A_7 &= \begin{pmatrix} 0&0&0\\0&0&1\\0&1&0\end{pmatrix},
      &A_8 &= \begin{pmatrix} 0&0&0\\0&0&i\\0&-i&0\end{pmatrix}.      
\end{array}
\end{equation}

The generalized special unitary algebra $\mathfrak{su}(2, 1)$ is a non-compact real form of the complex special linear algebra $\mathfrak{sl}_3(\mathbb{C})$.
The Lie group corresponding to $\mathfrak{su}(2, 1)$ is generalized special unitary group  
$\mathrm{SU}(2,1)$. In this section we use conjugations of $\SL_3(\C)$ and $\ssl_3(\C)$ that we denote by the same symbol $\sigma$. For $g\in \SL_3(\C)$ we
have $\sigma(g) = N\bar{g}^{-t}N^{-1}$ and for $x\in \ssl_3(\C)$ we have $\sigma(x) = -N\bar{x}^tN^{-1}$. Then $\mathrm{SU}(2,1)$, $\su$ are the sets of elements of
respectively $\SL_3(\C)$ and $\ssl_3(\C)$ that are fixed under $\sigma$.

In Theorems \ref{onesu21t} through \ref{levisu21t}, we classify the real subalgebras of $\mathfrak{su}(2,1)$, up to conjugation in $\mathrm{SU}(2,1)$.
The classification is presented and summarized in Tables \ref{onesu21}  through \ref{levisu21}. In the tables, subalgebras separated by a horizontal dashed line are conjugate in $\mathrm{SL}_3(\mathbb{C})$ (but not under conjugation in $\mathrm{SU}(2,1)$). Further,  $\u \sim \mathfrak{v}$   indicates that subalgebras  $\u$ and  $\mathfrak{v}$ are conjugate with respect to $\mathrm{SU}(2,1)$.
Conjugacy conditions for cases with parameters are established by direct computational analysis.

The tables are organized according to dimension and structure. They are structurally organized into  solvable, semisimple, and Levi decomposable subalgebras.

In the theorems of this section, we use the  basis of $\mathfrak{sl}_3(\mathbb{C})$ employed in  the classification of subalgebras of $\mathfrak{sl}_3(\mathbb{C})$ from \cite{dr16a, dr16b, dr18}, namely the basis 
$\{H_\alpha, H_\beta, X_\alpha, X_\beta,X_{\alpha+\beta}, $$Y_\alpha, $$ Y_\beta,$$Y_{\alpha+\beta}\}$ 
defined in Eq. \eqref{sl3basiss}. This is also a basis of the real Lie algebra $\mathfrak{sl}_3(\mathbb{R})$. 

\begin{rmk}\label{rem:H1su}
A computation with the algorithm of \cite{bg} shows that the first Galois
cohomology set $H^1 (\SL_3(\C),\sigma)$ consists of the classes of
$\diag(1,1,1)$ and $\diag(-1,-1,1)$. In particular, the first Galois cohomology
is not trivial. If we consider a subalgebra $\v$ and the Galois cohomology
$H^1(\mathcal{Z}_G(\v),\sigma)$ then it may happen that the latter set contains
classes $[c]$ such that $c$ is equivalent to $\diag(-1,-1,1)$ in
$Z^1(\SL_3(\C),\sigma)$. Such elements
do not correspond to real orbits (see Theorem \ref{thm:galco}). 
\end{rmk}

\begin{theorem}\label{onesu21t}
A complete list of inequivalent, one-dimensional real subalgebras of $\su$, up to conjugation in $\mathrm{SU}(2,1)$, is given in Table
\ref{onesu21}.
\end{theorem}
\begin{proof} 
For each one-dimensional subalgebra representative of $\mathfrak{sl}_3(\mathbb{C})$ from Table \ref{table1sl3C}, we determine whether its $\mathrm{SL}_3(\mathbb{C})$-orbit contains a real point. If not, we disregard the subalgebra. If so, we use Galois cohomology to determine the $\mathrm{SU}(2,1)$-orbits.

\noindent $\underline{\v=\langle X_\alpha\rangle}$:  We have $X_\alpha= \tfrac{1}{2}(A_3-i A_4) $ so that $\v$ is not real. Let
\begin{equation}\label{eq:1}
  h = \begin{pmatrix} 0&-1&0\\-1&0&0\\0&0&-1\end{pmatrix}.
\end{equation}
Then $h\cdot \v = \sigma(\v)$. Furthermore, $h$ is a cocycle and with 
\begin{equation}\label{eq:2}
g = \tfrac{1}{2}\begin{pmatrix} -1&1&\sqrt{2}\\ -1&1&-\sqrt{2}\\
  -\sqrt{2}&-\sqrt{2} & 0 \end{pmatrix}
\end{equation}
we have $g^{-1}\sigma(g) = h$. Hence $g\cdot \v$ is real.
In fact, $g\cdot \v$ is spanned by 
$x=gX_\alpha g^{-1}=\tfrac{1}{4}iA_1+\tfrac{1}{2}iA_2+\tfrac{1}{4}iA_4-\tfrac{1}{4} \sqrt{2} iA_6
-\tfrac{1}{4}\sqrt{2} iA_8$ so that $\sigma(x) = -x$. A real basis
element is $ix$, which is
\begin{equation}-\tfrac{1}{4}A_1-\tfrac{1}{2}A_2-\tfrac{1}{4}A_4+\tfrac{1}{4} \sqrt{2} A_6
+\tfrac{1}{4}\sqrt{2} A_8.\end{equation}
We have that $\mathcal{Z}_G(\v)=T\ltimes U$ where $U$ is unipotent and $T$ is a torus consisting of elements
$\diag(a,b,c)$ with $c=(ab)^{-1}$ (see Table \ref{tab:stab1}). The stabilizer of $g\cdot \v$ is $g\mathcal{Z}_G(\v)g^{-1} = S\ltimes U'$
where $S=gTg^{-1}$ and $U'=gUg^{-1}$. We have that $S$ consists of the elements $S(a,b)$ where
$$S(a,b) = \tfrac{1}{4}\begin{pmatrix} a+b+2c & a+b-2c & \sqrt{2}(a-b)\\
a+b-2c & a+b+2c & \sqrt{2}(a-b) \\ \sqrt{2}(a-b) & \sqrt{2}(a-b) & 2(a+b)\end{pmatrix},$$
where $c=(ab)^{-1}$. A calculation shows that $\sigma(S(a,b)) = S(\bar b^{-1}, \bar a^{-1})$.
By Lemma \ref{lem:T2} this implies that the first Galois cohomology
set of $S$ is trivial. By Lemma \ref{lem:sansuc} we see that the first Galois cohomology set 
of $gHg^{-1}$ is trivial as well. Hence, up to $\SU$-conjugacy, there is
one subalgebra in $\su$ which is $G$-conjugate to $\v$.  Denote 
\begin{equation}
    \v_{1,1}=\big\langle -A_1-2A_2-A_4+\sqrt{2} A_6+\sqrt{2} A_8 \big\rangle.
\end{equation}

\noindent $\underline{\v=\langle X_\alpha+X_\beta\rangle}$:  We have $X_\alpha+
X_\beta= \tfrac{1}{2}(A_3-i A_4)+\tfrac{1}{2}(A_7-iA_8) $, and
$\sigma(X_\alpha) = \tfrac{1}{2}(A_3+iA_4)+\tfrac{1}{2}(A_7+iA_8)$ so that $\v$ is
not real. Let
\begin{equation}\label{eq:eq4}
g_0 = \begin{pmatrix} 0&0&1\\0&1&0\\-1&0&0\end{pmatrix}.
\end{equation}
Then $g_0^{-1}(X_\alpha+X_\beta) g_0 = \sigma(X_\alpha+X_\beta)$ so that
$g_0^{-1} \cdot \v = \sigma(\v)$. Also $g_0\sigma(g_0)=1$ so that $g_0$ is a
cocycle. Let
\begin{equation}\label{eq:eq5}
h=\begin{pmatrix} \tfrac{1}{2}(1-i) & 0 & \tfrac{1}{2}(-1+i)\\
0&i&0\\ \tfrac{1}{2}(1-i) & 0 & \tfrac{1}{2}(1-i) \end{pmatrix}.
\end{equation}
Then $h\in \SL(3,\C)$ is such that $h^{-1}\sigma(h) = g_0$. Furthermore
$h\cdot \v$ is spanned by $A_3+A_4+A_7-A_8$ so it is indeed real.
The identity component of $\mathcal{Z}_G(\v)$ is of the form $T\ltimes U$ where $T$ is a
torus consisting of the elements $T(t) = \diag(t,1,t^{-1})$ and $U$ is
unipotent (Table \ref{tab:stab1}). The stabilizer of $h\cdot \v$ is $h\mathcal{Z}_G(\v)h^{-1}$. So its identity
component is of the form $S\ltimes U'$ where $S$ is a torus consisting of
the elements $S(t) = hT(t)h^{-1}$. A calculation shows that
\begin{equation}S(t) =\begin{pmatrix} \tfrac{t+t^{-1}}{2} & 0 & \tfrac{t-t^{-1}}{2}\\[2pt]
0&1&0 \\[2pt] \tfrac{t-t^{-1}}{2} & 0 & \tfrac{t+t^{-1}}{2} \end{pmatrix}.\end{equation}
Another calculation shows that $\sigma(S(t)) = S(\bar t)$. By Lemma
\ref{lem:T1}(1) 
the Galois cohomology of $S$ is trivial. By Lemma \ref{lem:sansuc} the Galois
cohomology of $h\mathcal{Z}_G(\v)^\circ h^{-1}$ is trivial as well. As the component group is
of order 3, Proposition \ref{prop:compodd} shows that the Galois cohomology of $h\mathcal{Z}_G(\v)h^{-1}$ is trivial
as well. So here we get one real algebra.
 Denote 
\begin{equation}
    \v_{1,2}=\big\langle A_3+A_4+A_7-A_8 \big\rangle.
\end{equation}

\noindent $\underline{\v=\langle X_\alpha+H_\alpha+2H_\beta\rangle}$:  We have
$a=X_\alpha+H_\alpha+2H_\beta= \tfrac{1}{2}(A_3-i A_4)-iA_1-2iA_2$, so that $\v$ is
not real. Let $h$ be as in \eqref{eq:1}. Then $h$ is a cocycle and it
maps $a$ to $-\sigma(a)$. We have $h=g^{-1}\sigma(g)$ with $g$ as in \eqref{eq:2}
and
\begin{equation}g\cdot a = \tfrac{1}{4}( 3iA_1+6iA_2-5iA_4-i\sqrt{2}A_6-i\sqrt{2}A_8).\end{equation}
So $g\cdot \v$ is real and spanned by $x=3A_1+6A_2-5A_4-\sqrt{2}A_6-\sqrt{2}A_8$.
The stabilizer $\mathcal{Z}_G(\v)$  is connected and equal to $T\ltimes U$ with $T$ being a torus consisting of
elements $T(t) = \diag(t,t,t^{-1})$ and $U$ is unipotent
(Table \ref{tab:stab1}). The stabilizer of
$g\cdot \v$ is $g\mathcal{Z}_G(\v)g^{-1}$. It is equal to $S\ltimes U'$,
where $S$ is a 1-dimensional torus.  We have
\begin{equation}S(t)=gT(t)g^{-1} = \begin{pmatrix} \tfrac{t+t^{-2}}{2} & \tfrac{t-t^{-2}}{2}& 0\\[2pt]
  \tfrac{t-t^{-2}}{2}& \tfrac{t+t^{-2}}{2}& 0\\[2pt] 0&0&t \end{pmatrix}.\end{equation}
A calculations shows that $\sigma(S(t)) = S(\bar t^{-1})$. By Lemma
\ref{lem:T1}(2) this implies that $H^1 S$ has
two elements with representatives $S(1)$, $S(-1)$. It happens that $S(-1)=h$.
So $S(-1)=g^{-1}\sigma(g)$. So the second real algebra in the complex orbit
is $g\cdot(g\cdot\v)$. It is spanned by $A_1+2A_2+7A_4-\sqrt{2}A_6+\sqrt{2}A_8$.
 Denote 
\begin{equation}
\begin{array}{lll}
 \v_{1,3}&=&\big\langle 3A_1+6A_2-5A_4-\sqrt{2}A_6-\sqrt{2}A_8 \big\rangle ~\text{and}\\[1.2ex]
    \v_{1,4}&=&\big\langle A_1+2A_2+7A_4-\sqrt{2}A_6+\sqrt{2}A_8 \big\rangle.
    \end{array}
\end{equation}

\noindent $\underline{\v_\lambda=\langle H_\lambda\rangle}$ with $H_\lambda =
H_\alpha+\lambda H_\beta$ and $\lambda\in \C$. We have that $\v_\lambda
\sim \v_\mu$ if and only if $\mu \in C_\lambda= \Big\{\lambda,\tfrac{1}{\lambda},1-\lambda,
\tfrac{1}{1-\lambda},\tfrac{\lambda}{\lambda-1},\tfrac{\lambda-1}{\lambda}\Big\}$.
The algebra $\v_\lambda$ is spanned by $A_1+\lambda A_2$, 
 and we denote
\begin{equation}
    \v_{1,5}^\lambda=\big\langle A_1+\lambda A_2\big\rangle.
\end{equation} (Furthermore, it is straightforward to establish that $\v_{1,5}^\lambda \sim \v_{1,5}^\eta$ if and only if $\eta=\lambda$ or $\eta =\frac{\lambda}{\lambda-1}$. It follows that   $\v_{1,5}^\lambda$ represents a non-equivalent subalgebra for each $\lambda\in [0, 2]$, and each subalgebra is included for $\lambda \in [0,2]$.)
Hence
$\sigma(\v_\lambda) = \v_{\bar\lambda}$. So $\sigma(\v_\lambda)$ is $\SL(3,\C)$-conjugate
to $\v_\lambda$ if and only if $\bar \lambda \in C_\lambda$. So we have the
following cases:
\begin{enumerate}
\addtolength{\itemsep}{3pt}
\item $\bar\lambda = \lambda$: this is the same as $\lambda\in \R$. 
\item $\bar \lambda = \tfrac{1}{\lambda}$: this is equivalent to $\lambda\in\C$
  with $|\lambda|=1$.
\item $\bar\lambda = 1-\lambda$: this amounts to $\lambda = \tfrac{1}{2}+iy$
  with $y\in \R$.
\item $\bar\lambda = \tfrac{\lambda}{\lambda-1}$: this is equivalent to
  $\lambda=x+iy$, $x,y\in \R$ with $x^2-2x+y^2=0$.
\item $\bar\lambda = \tfrac{1}{1-\lambda}$ or $\bar\lambda =
  \tfrac{\lambda-1}{\lambda}$. There are no $\lambda\in\C$ satisfying either
  of these conditions.
\end{enumerate}
So we have to deal with four cases. In the first case we have $\lambda\in\R$.
This again splits into sub-cases, because if  $\lambda = 0,1,-1,
\tfrac{1}{2},2$ the stabilizer of $\v_\lambda$ is different from the generic
case. However, note that the algebras with $\lambda=0,1$ are conjugate, and
the same holds for the algebras with $\lambda=-1,\tfrac{1}{2},2$. So we
distinguish three cases: $\lambda$ generic, $\lambda=0$, $\lambda=2$.
First suppose that $\lambda \neq
0,1,-1,\tfrac{1}{2},2$. Then the stabilizer of $\v_\lambda$ consists of
$T(a,b)=\diag(a,b,\tfrac{1}{ab})$. We have $\sigma(T(a,b)) = T(\bar{a}^{-1},
\bar{b}^{-1})$. It follows that $T$ is the direct product of two tori considered
in Lemma \ref{lem:T1}(2). So from that lemma it follows that
the first Galois cohomology set consists of the
classes of $T(1,1)$, $T(-1,1)$, $T(1,-1)$, $T(-1,-1)$. The first element is
the identity, hence corresponds to the subalgebra $\v_\lambda$ itself. For the
second element set
\begin{equation}g_1 = \begin{pmatrix} 0 & \tfrac{1}{2}\sqrt{2} & \tfrac{1}{2}\sqrt{2}\\[2pt]
  0 & \tfrac{1}{2}\sqrt{2}& -\tfrac{1}{2}\sqrt{2} \\[2pt] -1 & 0 & 0\end{pmatrix}.\end{equation}
Then $g_1^{-1} \sigma(g_1) = T(-1,1)$. We have that $g_1\cdot \v_\lambda$ is
spanned by $A_1 +2A_2-(2\lambda-1)A_4$, and we denote

\begin{equation}
 \s^\lambda_1 =\langle A_1 +2A_2-(2\lambda-1)A_4\rangle, ~ \lambda \in \mathbb{R}\setminus \{-1, 1, 2  \}.
\end{equation}
It is straightforward to show that
$\v_{1,5}^\lambda \sim 
\s^{\frac{\lambda-1}{\lambda}}_1$.
Thus, the  case of $\s^\lambda_1$ is redundant.

 Concerning the
third element set
\begin{equation}\label{eq:g2}
g_2 = \begin{pmatrix} \tfrac{1}{2}\sqrt{2} & 0 & -\tfrac{1}{2}\sqrt{2}\\[2pt]
  -\tfrac{1}{2}\sqrt{2} & 0 & -\tfrac{1}{2}\sqrt{2}\\[2pt]0&1&0\end{pmatrix}.\end{equation}
Then $g_2^{-1}\sigma(g_2) = T(1,-1)$. Furthermore, $g_2\cdot \v_\lambda$ is
spanned by $(1-\lambda)A_1-2(\lambda-1)A_2-(1+\lambda)A_4$.  Set
\begin{equation}
    \s^\lambda_2 =\langle (1-\lambda)A_1-2(\lambda-1)A_2-(1+\lambda)A_4 \rangle, ~\lambda \in \mathbb{R} \setminus \bigg\{0, \frac{1}{2}\bigg\}. 
\end{equation}
It is straightforward to show $\v^\lambda_{1,5} \sim \s^{-1/(\lambda-1)}_2$.  
 Hence, the subalgebra $\s^\lambda_2$ is redundant.

The cocycle $T(-1,-1)$ is not
equivalent to the trivial cocycle, and hence does not correspond to a real
subalgebra (see Remark \ref{rem:H1su}). 

Now suppose $\lambda=0$. In this case the stabilizer is $T\ltimes C$ (Table \ref{tab:stab1}), with 
$T$ as before (i.e., consisting of $T(a,b) = \diag(a,b,(ab)^{-1})$ for $a,b\in \C^\times$) and where $C$ is the component group consisting of the identity
and
\begin{equation}\label{udefined}
u = \begin{pmatrix} 0&1&0\\ 1&0&0\\ 0&0&-1\end{pmatrix}.\end{equation}
We have $\sigma(u)=u$ and $u^2=1$ so in particular $u$ is a cocycle.

We now compute the first Galois cohomology of the stabilizer using the procedure outlined at the end of Section \ref{sec:H1facts}. 
The set $H^1 (T,\sigma)$ is given above. We consider the right action of $C^\sigma=C$
on $H^1 (T,\sigma)$ defined by $[g]\cdot c = [c^{-1}g\sigma(c)]$.
We have $[ T(-1,1) ]\cdot u = [T(1,-1)]$ and $u$ leaves
$[T(1,1)]$, $[T(-1,-1)]$ invariant. Let $j : Z_G(\v_0) \to C$
denote the projection and $j_* : H^1 Z_G(\v_0) \to H^1 C$ the induced map
(note that $H^1 C = C$). It follows that $j_*^{-1}([1]) = \{ [T(1,1)], [T(-1,1)],
[T(-1,-1)]\}$. Now we consider the conjugation $\tau$ given by $\tau(g) = u\sigma(g)u$. We have
$\tau( T(a,b) ) = T(\bar b^{-1},\bar a^{-1})$. The Galois cohomology
$H^1 (T,\tau)$ is trivial (Lemma \ref{lem:T2}).
  Hence $j_*^{-1}([u]) = \{[u]\}$. Furthermore,
$H^1(\mathcal{Z}_G(\v_0),\sigma) = j_*^{-1}([1]) \cup j_*^{-1}([u])$. The
elements of $j_*^{-1}([1])$ yield the subalgebras $\v_0$, $g_1\cdot \v_0$.
For the cocycle $u$ set
\begin{equation}\label{eq:g3}
g_3 = \begin{pmatrix} -\tfrac{1}{2} & -\tfrac{1}{2} & -\tfrac{1}{2}\sqrt{2}\\[2pt]
  \tfrac{1}{2} & \tfrac{1}{2} & -\tfrac{1}{2}\sqrt{2}\\[2pt] \tfrac{1}{2}\sqrt{2}
  & -\tfrac{1}{2}\sqrt{2} & 0 \end{pmatrix}.
  \end{equation}
Then $g_3^{-1}\sigma(g_3) = u$ and $g_3\cdot \v_0$ is spanned by $A_5-A_7$,  and we denote
\begin{equation}
    \v_{1,6} =\langle A_5-A_7\rangle.
\end{equation}

Let $\lambda=2$. From Table \ref{tab:stab1} we see that the stabilizer is isomorphic to $\GL(2,\C)$ 
with conjugation $A\mapsto \overline{A}^{-t}$. A calculation with the algorithm from \cite{bg} shows that the first Galois cohomology
of this group consists of the classes of $\diag(1,1)$, $\diag(1,-1)$,
$\diag(-1,-1)$. Hence $H^1 \mathcal{Z}_G(\v_{2})$ consists of the classes of
$T(1,1)$, $T(1,-1)$, $T(-1,-1)$, with $T(a,b) = \diag(a,b,(ab)^{-1})$ as before. 
We get the subalgebras $\v_2$, $g_2\cdot \v_2$.

Now consider the second case where $\lambda\in \C$ is such that $|\lambda|=1$.
We assume $\lambda\not\in\R$. This splits into two cases. In the first one
we have $\lambda \neq \tfrac{1\pm i\sqrt{3}}{2}$. Let $g_0$, $h$ be as in \eqref{eq:eq4}, 
\eqref{eq:eq5}. 
Then $g_0 H_\lambda g_0^{-1} = -\lambda H_{\bar\lambda}$ (as $\bar\lambda =
\tfrac{1}{\lambda}$). Furthermore $g_0$ is a cocycle and $g_0 =
h^{-1}\sigma(h)$. Writing $\lambda = x+iy$ with $x^2+y^2=1$ and $y\neq 0$
(as $\lambda\not\in\R$) we see that $h\cdot \u_\lambda$ is spanned by
$A_1-A_2+\tfrac{y}{x-1}A_5$. The stabilizer of this algebra consists of the
matrices $h \diag(a,b,c) h^{-1}$ where $abc=1$.
A computation shows that it also consists of
\begin{equation}\label{eq:eqT}
T(a,c) = \begin{pmatrix} \tfrac{a+c}{2} & 0 & \tfrac{a-c}{2} \\[2pt]
  0&b&0\\[2pt] \tfrac{a-c}{2} & 0 & \tfrac{a+c}{2} \end{pmatrix}, \text{ with }
a\hspace{0.1em}b\hspace{0.1em}c=1.
\end{equation}
We have $\sigma( T(a,c) ) = T(\bar c^{-1},\bar a^{-1})$. Hence the stabilizer
has trivial Galois cohomology (Lemma \ref{lem:T2}). 
If $\lambda = \tfrac{1\pm i\sqrt{3}}{2}$ then we have the same
argument. In this case the identity component of the stabilizer is the same, hence its Galois
cohomology is trivial. The component group is of order 3, so the cohomology of the whole group is
trivial (Proposition \ref{prop:compodd}). 
We conclude that for $|\lambda|=1$ we get one series of subalgebras.  It is generated by $A_1-A_2+\frac{y}{x-1} A_5$, where $x, y\in \mathbb{R}$, $y\neq 0$, and $x^2+y^2=1$. This allows us to define
\begin{equation}
    \v_{1,7}^\lambda = \langle A_1-A_2+\lambda A_5 \rangle, ~\lambda \in \mathbb{R}\setminus \{0\}.
\end{equation}
Further, it is straightforward to show $\v^\lambda_{1,7}\sim \v^\eta_{1,7}$ if and only if $\eta =\pm \lambda$. It follows that   $\v_{1,7}^\lambda$ represents a non-equivalent subalgebra for each $\lambda\in (0, \infty)$, and each subalgebra is included for $\lambda \in (0, \infty)$.

In the third case, we have $\lambda = \tfrac{1}{2}+iy$ where we assume $y\neq \pm \tfrac{\sqrt{3}}{2}$. Set
\begin{equation}g_0 = \begin{pmatrix} 1&0&0\\0&0&i\\0&i&0\end{pmatrix}.\end{equation}
Then $g_0H_\lambda g_0^{-1} = H_{\bar\lambda}$ and $g_0$ is a cocycle. We have that $g_0 = 
g^{-1}\sigma(g)$  with 
\begin{equation}g=\begin{pmatrix} i&0&0\\[3pt] 0&\tfrac{1-i}{2} & -\tfrac{1+i}{2}\\[3pt] 0&-\tfrac{1+i}{2} & \tfrac{1-i}{2}\end{pmatrix}.\end{equation}
The subalgebra spanned by $gH_\lambda g^{-1}$ is also spanned by $u=A_1+\tfrac{1}{2} A_2+yA_8$. 
The stabilizer of $\v_\lambda$ in $\SL(3,\C)$ consists of $T(b,c)=g\,\diag(a,b,c)\,g^{-1}$ with $a=(bc)^{-1}$. Hence
\begin{equation}T(b,c) = \begin{pmatrix} a&0&0 \\[3pt] 0&\frac{b+c}{2} & i\tfrac{b-c}{2}\\[3pt]
0&-i\tfrac{b-c}{2} & \tfrac{b+c}{2}\end{pmatrix} \text{ with } ~a = (bc)^{-1}.\end{equation}
We have that $\sigma(T(b,c)) = T(\bar{c}^{-1},\bar{b}^{-1})$. Hence the Galois cohomology of this 
group is trivial (Lemma \ref{lem:T2}) and we get one class of subalgebras.
 Set 
\begin{equation}
    \s_3^\lambda = \bigg\langle A_1+\frac{1}{2} A_2+ \lambda A_8\bigg\rangle, ~ \lambda \in \mathbb{R}\setminus \{0\}.
\end{equation}
We have that $\v_{1,7}^\lambda \sim \s_3^{\lambda/2}$, hence this subalgebra is redundant.

In the fourth case we have $\lambda = x+iy$ with $x^2-2x+y^2=0$; we may assume $y\neq 0$ hence also 
$x\neq 0$. With 
\begin{equation}\label{eq:eq3}
g_0= \begin{pmatrix} 0&1&0\\1&0&0\\0&0&-1\end{pmatrix} \text{ and }~ g = \begin{pmatrix}
     -\tfrac{1}{2} & -\tfrac{1}{2} & -\tfrac{1}{2}\sqrt{2}\\[3pt]  \tfrac{1}{2} & \tfrac{1}{2} & -\tfrac{1}{2}\sqrt{2}\\[3pt]  \tfrac{1}{2}\sqrt{2} & -\tfrac{1}{2} \sqrt{2}& 0 
\end{pmatrix}
\end{equation}
we have $g_0H_\lambda g_0^{-1} = (\lambda-1)H_{\bar\lambda}$, $g_0$ is a cocycle and $g^{-1}\sigma(g) = g_0$. A short computation shows that the algebra spanned by $gH_\lambda g^{-1}$ is also spanned
by $A_1+2A_2+3A_4+\tfrac{y}{x}\sqrt{2}(A_5-A_7)$. The stabilizer of the latter algebra consists of 
$T(a,b)=g \,\diag(a,b,c)\,g^{-1}$ and we have
\begin{equation}T(a,b) = \begin{pmatrix} \tfrac{a+b+2c}{4} & \tfrac{-a-b+2c}{4} & \tfrac{\sqrt{2}(-a+b)}{4}\\
\tfrac{a-b+2c}{4} & \tfrac{a+b+2c}{4} & \tfrac{\sqrt{2}(a-b)}{4}\\
\tfrac{\sqrt{2}(-a+b)}{4}  & \tfrac{\sqrt{2}(a-b)}{4} & \tfrac{a+b}{4} \end{pmatrix}
\text{ with } c = (ab)^{-1}.\end{equation}
We have $\sigma(T(a,b)) = T(\bar{b}^{-1},\bar{a}^{-1})$. Hence the Galois cohomology of this group is trivial (Lemma \ref{lem:T2}). So we again get one class of algebras,  which is spanned
by $A_1+2A_2+3A_4+\tfrac{y}{x}\sqrt{2}(A_5-A_7)$ for 
$x, y\in \mathbb{R}$, $y\neq 0$, and $x^2+y^2=1$. 
This yields the subalgebra 
\begin{equation}
\s_{4}^\lambda=\big\langle A_1+2A_2+3A_4+\lambda\sqrt{2}(A_5-A_7) \big\rangle, ~ \lambda \in (-1,1)\setminus \{0\}.
\end{equation}
Then, $\s_4^\lambda \sim \v_{1,7}^\lambda$ under conjugation by
\begin{equation}
    \begin{pmatrix}
\frac{\sqrt{2}}{2} & -\frac{\sqrt{2}}{2}  & 0 \\[3pt]
-\frac{\sqrt{2}}{2}  & -\frac{\sqrt{2}}{2}  & 0\\[3pt]
0&0&-1
\end{pmatrix} \in \SU.
\end{equation}
Hence, the subalgebra $\s_4^\lambda$ is redundant.
With the given constraints on the parameters, one can also show directly that 
$\v_{1,5}^\lambda$, $\v_{1,6}$, and $\v_{1,7}^\lambda$
are pairwise inequivalent.
\end{proof}

\begin{theorem}\label{twosu21t}
A complete list of inequivalent two-dimensional real subalgebras of $\su$, up to conjugation in $\mathrm{SU}(2,1)$, is given in Table
\ref{twosu21}.
\end{theorem}
\begin{proof}
Similar to the previous proof, we consider each two-dimensional  subalgebra representative of $\mathfrak{sl}_3(\mathbb{C})$ from Table \ref{table2sl3C}.

\vspace{2mm}

\noindent $\underline{\v=\langle X_\alpha,X_{\alpha+\beta}\rangle}$: 
We have $X_\alpha = \tfrac{1}{2} (A_3-iA_4)$, $X_{\alpha+\beta} = -\tfrac{1}{2}(
A_5-iA_6)$ so that $\sigma(\v)$ is spanned by $\tfrac{1}{2} (A_3+iA_4)$,
$ -\tfrac{1}{2}(A_5+iA_6)$. A Gr\"obner basis computation shows that
there is no $g\in G$ with $g\cdot \v = \sigma(\v)$. Hence the orbit of
$\v$ does not contain real subalgebras.

\vspace{2mm}

\noindent $\underline{\v=\langle X_\alpha+X_\beta,X_{\alpha+\beta}\rangle}$: 
We have $X_\alpha+X_\beta = \tfrac{1}{2}(A_3-iA_4)+\tfrac{1}{2}(A_7-iA_8)$.
Let
\begin{equation} \label{eq:eqg0} g_0 = \begin{pmatrix} 0&0&-1\\ 0&1&0\\ 1&0&0\end{pmatrix}.\end{equation}
Then $g_0\cdot \v = \sigma(\v)$. Furthermore, $g_0\sigma(g_0) =1$, i.e., $g_0$ is a cocycle.
Let
\begin{equation}\label{eq:eqh}h=\begin{pmatrix} \tfrac{1}{2}(1+i) & 0 & \tfrac{1}{2}(1+i)\\
0&-i&0\\ \tfrac{1}{2}(-1-i) & 0 & \tfrac{1}{2}(1+i) \end{pmatrix}\end{equation}
then $h\in G$ and $h^{-1}\sigma(h) = g_0$. Hence $h\cdot \v$ is real, and it is
spanned by $A_3+A_4-A_7+A_8$, $A_1+A_2+A_6$. The identity component of the stabilizer of $\v$ in $G$ is $T\ltimes U$, where $T$ is a 1-dimensional torus and $U$ is the unipotent radical
(see Table \ref{tab:stab2}). 
We have that $T$ consists of the elements $\diag(a,1,a^{-1})$. Furthermore $S=hTh^{-1}$ consists
of 
\begin{equation}S(a) = \begin{pmatrix} \tfrac{a+a^{-1}}{2} & 0 & \tfrac{-a+a^{-1}}{2}\\[2pt]
0&1&0\\[2pt] \tfrac{-a+a^{-1}}{2} & 0 & \tfrac{a+a^{-1}}{2}\end{pmatrix}.\end{equation}
We have $\sigma(S(a)) = S(\bar a)$ so that the Galois cohomology of $S$ is trivial (Lemma \ref{lem:T1}(1)). It follows
that the Galois cohomology of the identity component of the stabilizer of $h\cdot \v$ is trivial.
Since the component group is of order 3, the same follows for the Galois cohomology of
the entire stabilizer (Proposition \ref{prop:compodd}).  Hence up $\SU$-conjugacy
there is one subalgebra of $\su$ that is $G$-conjugate to $\v$.  Denote
\begin{equation}
\v_{2,1}=\langle A_3+A_4-A_7+A_8, ~ A_1+A_2+A_6 \rangle.
\end{equation}

\noindent $\underline{\v=\langle X_\alpha, H_{\alpha}+2H_{\beta}\rangle}$: 
We have $\sigma(X_\alpha)=-Y_\alpha$ and $\sigma(H_\alpha+2H_\beta) = -H_\alpha-2H_\beta$. Let 
$u$ be as in \eqref{udefined}. 
Then $u\cdot \v = \sigma(\v)$. Let $g_3$ be as in \eqref{eq:g3}. Then $g_3^{-1}\sigma(g_3) = u$.
So $g_3\cdot \v$ is real. It is spanned by $A_1+2A_2-A_4+\sqrt{2}A_6-\sqrt{2}A_8$, 
$A_1+2A_2+3A_4$. The stabilizer of $\v$ is connected. Its reductive part is a torus consisting
of elements $T(a,b) = \diag(a,b,(ab)^{-1})$ (Table \ref{tab:stab2}). Write $S(a,b) = g_3T(a,b)g_3^{-1}$. Then
\begin{equation}S(a,b) = \begin{pmatrix} \tfrac{1}{4}a+\tfrac{1}{4}b +\tfrac{1}{2}(ab)^{-1} &
-\tfrac{1}{4}a-\tfrac{1}{4}b +\tfrac{1}{2}(ab)^{-1} & \tfrac{1}{4}\sqrt{2}(-a+b)\\[2pt]
-\tfrac{1}{4}a-\tfrac{1}{4}b +\tfrac{1}{2}(ab)^{-1} & \tfrac{1}{4}a+\tfrac{1}{4}b +\tfrac{1}{2}(ab)^{-1} &
\tfrac{1}{4}\sqrt{2}(a-b)\\[2pt] \tfrac{1}{4}\sqrt{2}(-a+b) & \tfrac{1}{4}\sqrt{2}(a-b) & \tfrac{1}{2}(a+b)\end{pmatrix}.
\end{equation}
A computation shows that $\sigma(S(a,b)) = S(\bar b^{-1}, \bar a^{-1})$. So by
Lemma \ref{lem:T2} the first Galois cohomology
of the reductive part of the stabilizer of $g_3\cdot \v$ is trivial. Therefore
Lemma \ref{lem:sansuc} shows that the first Galois cohomology of the 
entire stabilizer is trivial. So we obtain one real subalgebra in this case.
Denote
\begin{equation}
\v_{2,2} = \langle A_1+2A_2-A_4+\sqrt{2}A_6-\sqrt{2} A_8, ~A_1+2A_2+3A_4\rangle.
\end{equation}

\noindent $\underline{\v=\langle H_\alpha,H_{\beta}\rangle}$:
This algebra is spanned by $A_1,A_2$ so it is real already. We use the procedure outlined at the 
end of Section \ref{sec:H1facts} to compute the Galois cohomology of the stabilizer. 
The stabilizer $N$ has six components
and its identity component $N^\circ =T$ is a 2-dimensional torus consisting of $\diag(a,b,(ab)^{-1})$ (Table \ref{tab:stab2}). Let $C=N/N^\circ$ be the component group. The conjugation $\sigma$ acts trivially
on $C$. Hence $H^1(C,\sigma) = \{ [1], [c] \}$ where $c$ is an element of order 2. For $c$ we can take
the class of $g_0$, with $g_0$ as in \eqref{eq:eq3}. 
Then $g_0$ is a cocycle in $N$. We have that $H^1(T,\sigma)$ has 4 elements. The group $C^\sigma$ acts
on it with two orbits: one containing the identity (which is an orbit of size 3) and the other
containing $\diag(-1,-1,1)$. If we twist the conjugation by $g_0$ then the first Galois cohomology
set is trivial. It follows that $H^1(N,\sigma)$ has 3 elements, namely the classes of $1$, 
$\diag(-1,-1,1)$ and $g_0$. However, the cocycle $\diag(-1,-1,1)$ is not equivalent to the
identity in $Z^1(\SL(3,\C),\sigma)$. From this it follows that we get subalgebras corresponding to 
the class of the identity and the class of $g_0$. Let $g$ be as in \eqref{eq:eq3}. Then 
$g\cdot \v$ is spanned by $A_6-A_7$, $A_1+2A_2+3A_4$. (We remark that his result corresponds to the 
known fact that $\mathfrak{su}(2,1)$ has two Cartan subalgebras up to conjugacy.) 
Denote
\begin{equation}
\begin{array}{llll}
\v_{2,3}= \langle A_1, ~A_2\rangle ~\text{and}~
\v_{2,4} = \langle A_6-A_7, ~ A_1+2A_2+3A_4\rangle.
\end{array}
\end{equation}

\noindent $\underline{\v=\langle X_\alpha+X_\beta,H_\alpha+H_{\beta}\rangle}$:
Let $g_0$, $h$, $S(a)$ be as in the treatment of the algebra spanned by $X_\alpha+X_\beta$,
$X_{\alpha+\beta}$. Also here we have $g_0\cdot \v = \sigma(\v)$. Moreover, $h\cdot\v$ is spanned
by $A_3+A_4-A_7+A_8$, $A_5$. Also in this case the identity component of the stabilizer
is a 1-dimensional torus consisting of $\diag(a,1,a^{-1})$ times a unipotent group.
So the reductive part of the stabilizer of $h\cdot \v$ is the torus with elements $S(a)$.
Also in this case the component group is of order 3. So also in this case the Galois cohomology
is trivial. We get one subalgebra.
Denote
\begin{equation}
\v_{2,5} = \langle A_3+A_4-A_7+A_8, ~A_5\rangle.
\end{equation}

\noindent $\underline{\langle X_\alpha, Y_{\beta}\rangle}$, $\underline{\langle X_\alpha, -H_\alpha+H_\beta+3X_{\alpha+\beta}\rangle}$, 
$\underline{\langle X_\alpha, -2H_\alpha-H_\beta+3Y_{\beta}\rangle}$: for these algebras $\v$ there is no $g\in \SL(3,\C)$ with 
$g\cdot\v = \sigma(\v)$. Hence they do not yield subalgebras of $\mathfrak{su}(2,1)$.

\vspace{2mm}

\noindent $\underline{\v_\lambda=\langle X_\alpha, \lambda H_\alpha+(2\lambda+1)H_\beta\rangle}$:
in this case there is a $g\in \SL(3,\C)$ with $g\cdot \v_\lambda = \sigma(\v_\lambda)$ if and only if $\lambda=-\tfrac{1}{2}+iy$ where $y\in \R$.
So we suppose $\lambda = -\tfrac{1}{2}+iy$. Let $g,g_0$ be as in \eqref{eq:eq3}. Then $g_0\cdot \v_{\lambda }= \sigma(\v_{\bar\lambda})$.
Hence $g\cdot \v_{\lambda}$ is real. It is spanned by $A_1+2A_2-A_4+\sqrt{2}A_6-\sqrt{2}A_8$, $2yA_1+4yA_2+6yA_4-\sqrt{2}A_5+\sqrt{2}A_7$.
The stabilizer of $\v_{\lambda}$ is connected. Its reductive part is a torus consisting of the elements 
$\diag(a,b,(ab)^{-1})$. In exactly the same way as below \eqref{eq:eq3} we see that the first Galois cohomology of
the stabilizer of $g\cdot \v_{\lambda}$ is trivial. So here we get one series of algebras.
Let 
\begin{equation}
\v_{2,6}^\lambda=\big\langle A_1+2A_2-A_4+\sqrt{2}(A_6-A_8), ~2\lambda A_1+4\lambda A_2+6\lambda A_4-\sqrt{2}(A_5-A_7) \big\rangle,
\end{equation}
where  $\lambda \in \mathbb{R}$. 
Suppose that $\v_{2,6}^\lambda \sim \v_{2,6}^\eta$, with $\lambda, \eta \in \mathbb{R}$. Then, $2\lambda A_1+4\lambda A_2+6\lambda A_4-\sqrt{2}(A_5-A_7)$ and  
\begin{equation}
\alpha (A_1+2A_2-A_4+\sqrt{2}(A_6-A_8))+ 
\beta (2\eta A_1+4\eta A_2+6\eta A_4-\sqrt{2}(A_5-A_7))
\end{equation}
 must have the same Jordan Form for some $\alpha, \beta \in \mathbb{R}$, with  $\beta \neq 0$ (up to permutation of Jordan Blocks). Their respective Jordan Forms are
 \begin{small}
\begin{equation}
\begin{array}{lll}
       \begin{pmatrix}
-2i(i+2 \lambda) & 0 & 0\\
0 & 2i(i-2\lambda) & 0\\
0&0&8 i \lambda
\end{pmatrix} ~\text{and}~ 
\begin{pmatrix}
-2i \beta (i+2 \eta) & 0 & 0\\
0 & 2i\beta (i-2\eta) & 0\\
0&0&8 i \beta \eta
\end{pmatrix}.
\end{array}
\end{equation}\end{small}
Given that $\alpha, \beta,  \lambda, \eta \in \mathbb{R}$, with  $\beta 
\neq 0$, the only way for the Jordan Forms to coincide (up to permutation of Jordan Blocks) is if $\eta =\pm \lambda$. However, one can show by direct computation that
$\v_{2,6}^\lambda$ and $\v_{2,6}^{-\lambda}$ are not equivalent for $\lambda \neq 0$. 
Hence, $\v_{2,6}^\lambda \sim \v_{2,6}^\eta$ if and only if $\eta = \lambda$.
\end{proof}

\begin{theorem}\label{threesu21t}
A complete list of inequivalent, three-dimensional solvable real subalgebras of $\su$, up to conjugation in $\mathrm{SU}(2,1)$, is given in Table
\ref{threesu21}.
\end{theorem}
\begin{proof}
\noindent $\underline{\v=\langle X_\alpha, X_\beta, X_{\alpha+\beta}\rangle}$: Let $g_0$, $h$ be as in 
\eqref{eq:eq4}, \eqref{eq:eq5}. Then $g_0^{-1}\cdot \v = \sigma(\v)$ and $h^{-1}\sigma(h)=g_0$. The algebra $h\cdot \v$ is spanned by
$A_4-A_8$, $A_3+A_7$, $A_1+A_2-A_6$. The reductive part of the stabilizer of $\v$ is a torus consisting of 
$\diag(a,b,c)$ with $abc=1$. Hence the reductive part of the stabilizer of $h\cdot \v$ is the torus with elements
$T(a,c)$ as in \eqref{eq:eqT}. Its first Galois cohomology set is trivial. Hence the first Galois cohomology set of
the stabilizer of $h\cdot \v$ is trivial. It follows that we get one algebra.
Denote
\begin{equation}
\v_{3,1} = \langle A_4-A_8, ~A_3+A_7,~ A_1+A_2-A_6\rangle.
\end{equation}

\noindent $\underline{\v=\langle X_\alpha+X_\beta, X_{\alpha+\beta}, H_{\alpha}+H_{\beta}\rangle}$: Let $g_0$, $h$ be as in 
\eqref{eq:eq4}, \eqref{eq:eq5}. Then $g_0^{-1}\cdot \v = \sigma(\v)$ and $h^{-1}\sigma(h)=g_0$. The algebra $h\cdot \v$ is spanned by
$A_3+A_4+A_7-A_8$, $A_1+A_2-A_6$, $A_5$. The reductive part of the identity component of the stabilizer of $\v$ is a torus consisting
of $\diag(a,1,a^{-1})$. The component group has order 3. So by the same analysis as for the subalgebra $\langle X_\alpha+X_\beta\rangle$
we see that the first Galois cohomology set of the stabilizer of $g_0\cdot\v$ is trivial.
Denote
\begin{equation}
\v_{3,2} = \langle A_3+A_4+A_7-A_8, ~A_1+A_2-A_6, A_5\rangle.
\end{equation}

\noindent $\underline{\v=\langle X_\alpha, H_{\alpha}, H_{\beta}\rangle}$: Let $g$, $g_0$ be as in 
\eqref{eq:eq3}. Then $g_0^{-1}\cdot \v = \sigma(\v)$ and $g^{-1}\sigma(g)=g_0$. The algebra $g\cdot \v$ is spanned by 
$4A_4-\sqrt{2}A_6+\sqrt{2}A_8$, $A_5-A_7$, $A_1+2A_2+3A_4$. The stabilizer of $\v$ is connected. Its reductive part is a
torus consisting of $\diag(a,b,c)$ with $abc=1$. As seen below \eqref{eq:eq3} this implies that the first Galois cohomology set
of the stabilizer of $g\cdot \v$ is trivial. 
Denote
\begin{equation}
\v_{3,3} = \langle 4A_4-\sqrt{2}A_6+\sqrt{2}A_8, ~A_5-A_7, ~A_1+2A_2+3A_4\rangle.
\end{equation}

For all other 3-dimensional solvable subalgebras $\v$ of $\mathfrak{sl}(3,\mathbb{C})$ we have that $\v$ and $\sigma(\v)$ are not conjugate under $\SL(3,\C)$. 
So they do not yield subalgebras of $\mathfrak{su}(2,1)$.
\end{proof}

\begin{theorem}\label{threesemisu21t}
A complete list of inequivalent, semisimple real subalgebras of $\su$, up to conjugation in $\mathrm{SU}(2,1)$, is given in Table
\ref{threesemisu21}.
\end{theorem}
\begin{proof}
\noindent $\underline{\v=\langle X_\alpha+X_\beta, 2Y_{\alpha}+2Y_\beta, 2H_\alpha+2H_{\beta}\rangle}$:
Let $g= \diag(1,-1,-1)$, then $g^{-1}\cdot \v = \sigma(\v)$. Let $g_2$ be as in \eqref{eq:g2}, then $g_2^{-1}\sigma(g_2)= g$. Hence $g_2\cdot \v$ is real;
it is spanned by $A_4, A_6, A_7$. The stabilizer of $\v$ in $\SL(3,\C)$ is $Z=\mathrm{PSL}(2,\C)\times \{\diag(\omega,\omega,\omega)\}$ with $\omega^3=1$.
The stabilizer of $g_2\cdot\v$ is $g_2Zg_2^{-1}$, which has the same direct product decomposition. The first Galois cohomology set of the second 
factor is trivial. The Lie algebra of the first factor is $g_2\cdot \v$. 
A computation with the algorithm of \cite{bg} shows that the first Galois cohomology set of the first factor consists of the
classes $[1]$ and $[\diag(-1,-1,1)]$. The second class is not trivial in $H^1(\SL(3,\C),\sigma)$, and hence does not correspond to a real subalgebra.
We see that we get a single real subalgebra.
Denote
\begin{equation}
\v_{3,4} = \langle A_4,~ A_6, ~A_7\rangle.
\end{equation}

\noindent $\underline{\v=\langle X_{\alpha+\beta}, Y_{\alpha+\beta}, H_\alpha+H_{\beta}\rangle}$:
In this case we have $\sigma(\v)=\v$ so that $\v$ is real. A real basis of it is $A_1+A_2, A_5, A_6$. Let $Z$ denote the stabilizer of $\v$ in $\SL(3,\C)$.
From Table \ref{tab:stab4} we see that it is isomorphic to $\GL(2,\C)$. A computation with the algorithm from \cite{bg} shows that 
$H^1(Z,\sigma)$ consists of $[1]$, $[\diag(-1,-1,1)]$, $[\diag(1,-1,-1)]$. The first element corresponds to $\u$ itself. The second element does not correspond
to any real subalgebra. For the third element we let $g_2$ be as in \eqref{eq:g2}. Then $g_2^{-1}\sigma(g_2)= \diag(1,-1,-1)$. Hence $g_2\cdot\v$ is a second 
real subalgebra. It is spanned by $A_1, A_3, A_4$. 
Denote
\begin{equation}
\begin{array}{lllll}
\v_{3,5} = \langle A_1+A_2, ~A_5, ~A_6\rangle ~\text{and}~
\v_{3,6} = \langle A_1, ~A_3, ~A_4\rangle.
\end{array}
\end{equation}
\end{proof}

\begin{theorem}\label{fourfivesu21t}
A complete list of inequivalent, four- and five-dimensional solvable real subalgebras of $\su$, up to conjugation in $\mathrm{SU}(2,1)$, is given in Table
\ref{fourfivesu21}.
\end{theorem}
\begin{proof}
\noindent $\underline{\v=\langle X_\alpha, X_\beta, X_{\alpha+\beta}, H_\alpha+H_\beta \rangle}$: Let $g_0$, $h$ be as
in \eqref{eq:eqg0}, \eqref{eq:eqh} respectively. Then $\sigma(\v) = g_0\cdot \v$, $g_0$ is a cocycle and $h^{-1}\sigma(h) = g_0$.
Hence $h\cdot \v$ is real. In fact, it is spanned by $A_1+A_2+A_6, A_5, A_3-A_7, A_4+A_8$. The stabilizer of $\v$ in
$\SL(3,\C)$ is connected. Its reductive part is a 2-dimensional torus consisting of $T(a,c) = \diag(a,b,c)$ with $b=(ac)^{-1}$.
So the reductive part of the stabilizer of $h\cdot \v$ consists of $S(a,c) = hT(a,c)h^{-1}$. We have 
\begin{equation}
    S(a,c) = \begin{pmatrix} \tfrac{a+c}{2} & 0 & \tfrac{-a+c}{2}\\
0&b&0 \\ \tfrac{-a+c}{2} & 0 & \tfrac{a+c}{2}\end{pmatrix} \text{ with } b=(ac)^{-1}.
\end{equation}
A computation shows that $\sigma(S(a,c)) = S(\bar c^{-1}, \bar a^{-1})$. Therefore the first Galois cohomology set of the reductive part,
and hence of the stabilizer, is trivial (Lemmas \ref{lem:T2}, \ref{lem:sansuc}).
So we get one subalgebra. Denote \begin{equation}\v_{4,1} = \langle 
A_1+A_2+A_6, A_5, A_3-A_7, A_4+A_8\rangle.\end{equation}

\vspace{2mm}

\noindent $\underline{\v=\langle X_\alpha, X_\beta, X_{\alpha+\beta}, H_\alpha-H_\beta \rangle}$: Let $g_0$, $h$ be as in the previous case. Then $\sigma(\v) = g_0^{-1}\cdot \v$ and the stabilizer of
$\v$ in $\SL(3,\C)$ is exactly the same as in the previous case. So again we get one algebra,
\begin{equation}\v_{4,2} = \langle A_1-A_2, A_1+A_2+A_6, A_3-A_7, A_4+A_8\rangle.\end{equation}

For the other solvable, four-dimenasional  subalgebras $\v$ of $\sll(3,\C)$ in Table \ref{table4sl3C} we have that $\v$ and $\sigma(\v)$ are not conjugate
under $\SL(3,\C)$ and hence do not correspond to a subalgebra of $\su$.
We now consider the single five-dimensional, solvable subalgebra of $\mathfrak{sl}_3(\mathbb{C})$ in Table \ref{table4sl3C}.

\noindent $\underline{\v=\langle X_\alpha, X_\beta, X_{\alpha+\beta}, H_\alpha, H_\beta \rangle}$: Let $g_0$, $h$ be as
in \eqref{eq:eqg0}, \eqref{eq:eqh} respectively. Then $\sigma(\v) = g_0^{-1}\cdot \v$, $g_0$ is a cocycle and $h^{-1}\sigma(h) = g_0$. Again the stabilizer and hence the analysis is exactly the same as in
the first case above. 
So we get one subalgebra.  Denote \begin{equation}\v_{5,1} = \langle A_1-A_2,  A_1+A_2+A_6, A_5, A_3-A_7, A_4+A_8\rangle.\end{equation}

For all other 4-dimensional solvable subalgebras $\v$ of $\mathfrak{sl}(3,\mathbb{C})$ we have that $\v$ and $\sigma(\v)$ are not conjugate under $\SL(3,\C)$. So they do not yield subalgebras of $\mathfrak{su}(2,1)$.
\end{proof}

\begin{theorem}\label{levisu21t}
A complete list of inequivalent, Levi decomposable  real subalgebras of $\mathfrak{su}(2,1)$, up to conjugation in $\mathrm{SU}(2,1)$, is given in Table \ref{levisu21}.
\end{theorem}
\begin{proof}
Among the  Levi decomposable subalgebras $\v$ of $\sll(3,\C)$ there is only one with the property that $\v$ and $\sigma(\v)$ are $\SL(3,\C)$-conjugate.
This is the four-dimensional subalgebra $\v = \langle X_{\alpha+\beta}, Y_{\alpha+\beta}, H_\alpha, H_\beta\rangle$. In fact we have $\v=\sigma(\v)$. A real basis of $\v$ is 
$A_1, A_2, A_5, A_6$. The stabilizer of $\v$ is the same as the stabilizer of the 3-dimensional subalgebra $\langle X_{\alpha+\beta}, Y_{\alpha+\beta}, 
H_\alpha+H_\beta\rangle$. Hence also the Galois cohomology is the same. So we get one more real subalgebra, spanned by $A_1, A_2, A_3, A_4$.
 Denote
\begin{equation}
\v_{4,3} = \langle A_1, ~A_ 2, ~A_5, ~A_6 \rangle ~\text{and}~
\v_{4,4} = \langle A_1, ~A_ 2, ~A_3, ~A_4 \rangle.
\end{equation}
\end{proof}

\begin{table}[H]
\renewcommand{\arraystretch}{1.6} 
\caption{One-Dimensional Real Subalgebras of $\mathfrak{su}(2,1)$ from Theorem \ref{onesu21t}.} \label{onesu21}
\centering
\scalebox{0.95}{
\begin{tabular}{|c|c|}
\hline
Subalgebra Representative & Conditions \\
\hline \hline
$\v_{1,1}=\big\langle -A_1-2A_2-A_4+\sqrt{2} A_6+\sqrt{2} A_8 \big\rangle$ &\\
\hline
$\v_{1,2}=\big\langle A_3+A_4+A_7-A_8 \big\rangle$ & \\
\hline
$\v_{1,3}= \big\langle 3A_1+6A_2-5A_4-\sqrt{2}A_6-\sqrt{2}A_8 \big\rangle$ & \\
\hdashline 
$\v_{1,4}= \big\langle A_1+2A_2+7A_4-\sqrt{2}A_6+\sqrt{2}A_8\big\rangle$ & \\
\hline
$\v_{1,5}^\lambda = \big\langle A_1+\lambda A_2\big\rangle$ & $\lambda\in \R$, ~$ \v_{1,5}^\lambda \sim \v_{1,5}^\eta$ iff $\eta=\lambda$ or  $\frac{\lambda}{\lambda-1}$\\
 & (alternatively, $\v_{1,5}^\lambda$ with $\lambda\in [0,2]$)\\
\hdashline
$\v_{1,6} = \big\langle A_5-A_7\big\rangle$ & \\
\hdashline
$\v_{1,7}^\lambda = \big\langle A_1-A_2+\lambda A_5\big\rangle $ & $\lambda\in \R\setminus \{0\}$, ~ $\v_{1,7}^\lambda \sim \v_{1,7}^\eta$ iff $\eta =\pm \lambda$\\
 & (alternatively $\v_{1,7}^\lambda$ with $\lambda\in (0, \infty))$\\
\hline
\end{tabular} }
\end{table}

\begin{table}[H]
\renewcommand{\arraystretch}{1.6} 
\caption{Two-Dimensional Real Subalgebras of $\mathfrak{su}(2,1)$ from Theorem \ref{twosu21t}.} \label{twosu21}
\centering
\scalebox{0.87}{
\begin{tabular}{|c|c|}
\hline
 Subalgebra Representative & Conditions 
 \\
\hline \hline
 $\v_{2,1}=\langle A_3+A_4-A_7+A_8,  A_1+A_2+A_6 \rangle$  & 
 \\
\hline
$\v_{2,2} = \langle A_1+2A_2-A_4+\sqrt{2}A_6-\sqrt{2}A_8, A_1+2A_2+3A_4\rangle$ &
\\
\hline
$\v_{2,3}= \langle A_1, A_2\rangle$ & 
\\
\hdashline
$\v_{2,4} = \langle A_6-A_7, A_1+2A_2+3A_4\rangle$ &
\\
\hline
$\v_{2,5}=\langle A_3+A_4-A_7+A_8, A_5\rangle$ & 
\\
\hline
$\v_{2,6}^\lambda = \langle A_1+2A_2-A_4+\sqrt{2}A_6-\sqrt{2}A_8, 2\lambda A_1+4\lambda A_2+6\lambda A_4-\sqrt{2}A_5+\sqrt{2}A_7\rangle$ &  $\lambda \in\R$,\\
&  $\v_{2,6}^\lambda \sim \v_{2,6}^\eta$ iff $\eta=\lambda$\\
\hline
\end{tabular} }
\end{table}

\begin{table}[H]
\renewcommand{\arraystretch}{1.6} 
\caption{Three-Dimensional Solvable Real Subalgebras of $\mathfrak{su}(2,1)$ from Theorem \ref{threesu21t}.} \label{threesu21}
\vspace{2pt}
\centering
\scalebox{0.90}{
\begin{tabular}{|c|}
\hline
Subalgebra Representative \\
\hline \hline
$\v_{3,1}=\langle A_4-A_8, A_3+A_7, A_1+A_2-A_6\rangle$ 
 \\
\hline
$\v_{3,2} = \langle A_3+A_4+A_7-A_8, A_1+A_2-A_6, A_5\rangle$ \\
\hline
$\v_{3,3} = \langle 4A_4-\sqrt{2}A_6+\sqrt{2}A_8, A_5-A_7, A_1+2A_2+3A_4\rangle $\\
\hline
\end{tabular} }
\end{table}

\begin{table}[H]
\renewcommand{\arraystretch}{1.6} 
\caption{Semisimple Real Subalgebras of $\mathfrak{su}(2,1)$ from Theorem \ref{threesemisu21t}.} \label{threesemisu21}
\centering
\scalebox{0.98}{
\begin{tabular}{|c|}
\hline
 Subalgebra Representative \\
\hline \hline
 $\v_{3,4}=\langle A_4, A_6, A_7\rangle$ \\
\hline
$\v_{3,5}=\langle A_1+A_2, A_5, A_6\rangle$ \\
\hdashline
$\v_{3,6}=\langle A_1, A_3, A_4\rangle$\\
\hline
\end{tabular} }
\end{table}

\begin{table}[H]
\renewcommand{\arraystretch}{1.6} 
\caption{Four- and Five-Dimensional Solvable Real Subalgebras of $\mathfrak{su}(2,1)$ from Theorem \ref{fourfivesu21t}.} \label{fourfivesu21}
\centering
\vspace{2pt}
\scalebox{0.98}{
\begin{tabular}{|c|c|c|c|}
\hline
Dim. & Subalgebra Representative \\
\hline \hline
4& $\v_{4,1} = \langle A_1+A_2+A_6, A_5, A_3-A_7, A_4+A_8\rangle$\\
\hline
4 & $\v_{4,2} = \langle A_1-A_2, A_1+A_2+A_6, A_3-A_7, A_4+A_8\rangle$ \\
\hline
5& $\v_{5,1} = \langle A_1-A_2, A_1+A_2+A_6, A_5, A_3-A_7, A_4+A_8\rangle$
\\
\hline
\end{tabular} }
\end{table}

\begin{table}[H]
\renewcommand{\arraystretch}{1.6} 
\caption{Levi Decomposable  Real Subalgebras of $\mathfrak{su}(2,1)$ from Theorem \ref{levisu21t}.} \label{levisu21}
\centering
\vspace{2pt}
\scalebox{0.98}{
\begin{tabular}{|c|c|c|c|}
\hline
Dim. & Subalgebra Representative \\
\hline \hline
 4&  $\v_{4,3}=\langle A_1, A_ 2, A_5, A_6 \rangle$
 \\
\hdashline 
4& $\v_{4,4}=\langle A_1, A_ 2, A_3, A_4 \rangle$
\\
\hline
\end{tabular} }
\end{table}

\section{The real subalgebras of $\suc$}

The special unitary algebra \(\mathfrak{su}(3)\) is the real Lie algebra of traceless, skew-Hermitian \(3 \times 3\) complex matrices.
A basis is given by

\begin{equation}
\arraycolsep=1.5pt\def\arraystretch{1.28}
\begin{array}{llllllllllll}
B_1 &=&
\begin{pmatrix}
0 & i & 0 \\
i & 0 & 0 \\
0 & 0 & 0
\end{pmatrix}, &
B_2 &=&
\begin{pmatrix}
0 & 1 & 0 \\
-1 & 0 & 0 \\
0 & 0 & 0
\end{pmatrix},&
B_3 &=& 
\begin{pmatrix}
0 & 0 & i \\
0 & 0 & 0 \\
i & 0 & 0
\end{pmatrix}, \\
B_4 &=& 
\begin{pmatrix}
0 & 0 & 1 \\
0 & 0 & 0 \\
-1 & 0 & 0
\end{pmatrix}, &
B_5 &=& 
\begin{pmatrix}
0 & 0 & 0 \\
0 & 0 & i \\
0 & i & 0
\end{pmatrix}, &
B_6 &=&
\begin{pmatrix}
0 & 0 & 0 \\
0 & 0 & 1 \\
0 & -1 & 0
\end{pmatrix}, \\
B_7 &=& 
\begin{pmatrix}
i & 0 & 0 \\
0 & -i & 0 \\
0 & 0 & 0
\end{pmatrix}, &
B_8 &=& 
\begin{pmatrix}
0 & 0 & 0 \\
0 & i & 0 \\
0 & 0 & -i
\end{pmatrix}.
\end{array}
\end{equation}

The  special unitary algebra $\mathfrak{su}(3)$ is the compact real form of the complex special linear algebra $\mathfrak{sl}_3(\mathbb{C})$.
The Lie group corresponding to $\mathfrak{su}(3)$ is the special unitary group  
$\mathrm{SU}(3)$. In this section we use the conjugations $\sigma : \SL_3(\C)\to \SL_3(\C)$, $\sigma(g) = \bar{g}^{-t}$ and $\sigma : \ssl_3(\C)\to \ssl_3(\C)$,
$\sigma(x) = -\bar{x}^t$. Then $\mathrm{SU}(3)$ is the group of fixed points of $\SL_3(\C)$ under $\sigma$ and $\mathfrak{su}(3)$ is the algebra of
fixed points of $\ssl_3(\C)$ under $\sigma$. 

We state a theorem concerning the subalgebras of $\suc$. 
Although the result seems to be well-known, we have not found a proof in the existing literature.
It will allow us to significantly shorten our computations. 

\begin{theorem}\label{compactsimplification}
Let $\g$ be a compact semisimple Lie algebra. Let $\u\subset \g$ be a subalgebra, then $\u$ is
reductive in $\g$ (i.e., $\u=\s\oplus \c$ where $\s$ is semisimple and $\c$ is the centre 
consisting of semisimple elements). 
\end{theorem}

\begin{proof}
According to \cite[\S1, Proposition 2(b)]{bou9} $\u$ is compact. By \cite[\S1 Proposition 1(iv)]{bou9}
$\u$ is reductive. By the same proposition $\ad x$ for $x\in \g$ is semisimple. So also 
$\ad x$ for $x\in \c$ is semisimple. 
\end{proof}

\begin{rmk}\label{rem:H1su3}
A computation with the algorithm of \cite{bg} shows that $H^1( \SL_3(\C),\sigma)$ consists of the classes of the identity and of $\diag(-1,1,-1)$.
So if a class in the first Galois cohomology of a stabilizer is equal in $H^1(\SL_3(\C),\sigma)$ to the class of the latter element, then this class
does not correspond to a real subalgebra. 
\end{rmk}

In Theorem \ref{su3}, we classify the real subalgebras of $\mathfrak{su}(3)$, up to conjugation in $\mathrm{SU}(3)$.
The classification is presented and summarized in Table \ref{susa}. In the table,  $\u \sim \mathfrak{v}$   indicates that subalgebras  $\u$ and  $\mathfrak{v}$ are conjugate with respect to $\mathrm{SU}(3)$.
Conjugacy conditions for cases with parameters are established by direct computational analysis.

In the theorems of this section, we use the  basis of $\mathfrak{sl}_3(\mathbb{C})$ employed in  the classification of subalgebras of $\mathfrak{sl}_3(\mathbb{C})$ from \cite{dr16a, dr16b, dr18}, namely the basis 
$\{H_\alpha, H_\beta, X_\alpha, X_\beta,X_{\alpha+\beta}, $$Y_\alpha, $$ Y_\beta,$$Y_{\alpha+\beta}\}$ 
defined in Eq. \eqref{sl3basiss}. This is also a basis of the real Lie algebra $\mathfrak{sl}_3(\mathbb{R})$. 

\begin{theorem}\label{su3}
A complete list of inequivalent, real subalgebras of $\suc$, up to conjugation in $\mathrm{SU}(3)$, is given in Table
\ref{susa}.
\end{theorem}
\begin{proof} 
The elements of $\mathfrak{su}(3)$ are skew-Hermitian, and hence by elementary linear algebra they are semisimple and have purely imaginary eigenvalues.
Together with Theorem \ref{compactsimplification} this means that we only have to consider the subalgebras of $\ssl(3,\C)$ that are reductive in 
$\ssl(3,\C)$ and have a basis consisting of semisimple elements that have purely imaginary eigenvalues. A quick check shows that only the subalgebras
remain that we consider below.

\noindent $\underline{\w_\lambda=\langle H_\lambda\rangle}$ with $H_\lambda =
H_\alpha+\lambda H_\beta$ and $\lambda\in \C$. It is straghtforward to see that $\w_\lambda$ has a basis element having purely imaginary
eigenvalues if and only if $\lambda$ is real. So we only have to consider that case.

Here we use exactly the same analysis as
for $\su$. For $\lambda$ real we have that $\w_\lambda$ is real and spanned by 
$B_7+\lambda B_8$. If $\lambda$ is generic, then the stabilizer consists of
$T(a,b)=\diag(a,b,(ab)^{-1})$. We have that $\sigma(T(a,b)) = T(\bar a^{-1},\bar b^{-1})$. 
So Lemma \ref{lem:T1} shows that the first Galois cohomology set consists of the classes of the identity and 
of $g_2= \diag(1,-1,-1)$, $g_3=\diag(-1,1,-1)$, $g_4=\diag(-1,-1,1)$. Set 
\begin{equation}h_2=\begin{pmatrix} 0&1&0\\0&0&-1\\-1&0&0\end{pmatrix},\,
h_3= \begin{pmatrix} 1&0&0\\0&1&0\\0&0&1\end{pmatrix},\,
h_4= \begin{pmatrix} 0&0&1\\ 1&0&0\\ 0&1&0\end{pmatrix}.\end{equation}
Then $h_i^{-1} g_i\sigma(h_i) = \diag(-1,1,-1)$. It follows that all three classes are non-trivial
in $H^1(\SL(3,\C),\sigma)$. Hence, here we just get one real subalgebra. 

If $\lambda=0$ then using the same analysis as for $\su$ we see that $H^1(\mathcal{Z}_G(\w_0),\sigma)$ 
consists of the classes of the identity, $g_3$, $g_4$ and $u$  ($u$ from Eq. \eqref{udefined}). Above $g_3$, $g_4$ have been
shown to be nontrivial. For $u$ we set 
\begin{equation}h= \begin{pmatrix} \tfrac{1}{2}\sqrt{2} & -\tfrac{1}{2}\sqrt{2} & 0\\
-\tfrac{1}{2}\sqrt{2} & -\tfrac{1}{2}\sqrt{2} & 0\\ 0&0&-1\end{pmatrix}.\end{equation}
Then $hu\sigma(h)^{-1} = \diag(-1,1,-1)$ and we see that $u$ is nontrivial as well.
So we get just one algebra. 

If $\lambda=2$ then again using the same analysis as for $\su$ we get that the first Galois
cohomology set consists of the classes of the identity, $T(1,-1)$, $T(-1,-1)$. The last two
are nontrivial. Hence also here we get just one algebra. 

Summarizing: for $\lambda\in \R$ we get one algebra spanned by $B_7+\lambda B_8$. Denote 
\begin{equation} \w^\lambda_1=\langle  B_7+\lambda B_8\rangle, ~\lambda \in \mathbb{R}.\end{equation}

\noindent $\underline{\w=\langle H_\alpha,H_\beta\rangle}$. Here we use the well-known fact that $\suc$ has a single Cartan subalgebra up to conjugacy.
So we get one algebra spanned by $B_7, B_8$.  Denote 
\begin{equation}\w_2=\langle  B_7, B_8 \rangle.\end{equation}

\vspace{2pt}

\noindent $\underline{\w=\langle X_{\alpha+\beta}, Y_{\alpha+\beta}, H_\alpha+H_\beta\rangle}$. This algebra is real and spanned by $B_3,B_4,B_7+B_8$.
Using the algorithm of \cite{bg} we computed that the first Galois cohomology set of the stabilizer consists of the classes of 
$\diag(1,1,1)$, $\diag(-1,1,-1)$, $\diag(-1,-1,1)$. Above we have seen that the last two are nontrivial in $H^1( \SL(3,\C),\sigma)$. So we get one algebra.  Denote 
\begin{equation}\w_{3, 1}=\langle  B_3,B_4,B_7+B_8\rangle.\end{equation}

\vspace{2pt}

\noindent $\underline{\w=\langle X_{\alpha}+X_\beta, 2Y_\alpha+2Y_\beta, 2H_\alpha+2H_\beta\rangle}$. This algebra is real and spanned by 
$B_1+B_5$, $B_2+B_6$, $B_7+B_8$. The first Galois cohomology set of its stabilizer consists of the classes of $\diag(1,1,1)$, $\diag(-1,1,-1)$. The last
one is nontrivial in $H^1( \SL(3,\C),\sigma)$. So we get one algebra. Denote 
\begin{equation} \w_{3, 2}=\langle  B_1+B_5, B_2+B_6, B_7+B_8\rangle.\end{equation}

\vspace{2pt}

\noindent $\underline{\w=\langle X_{\alpha+\beta}, Y_{\alpha+\beta}, H_\alpha+H_\beta\rangle \oplus \langle H_\alpha-H_\beta\rangle}$. This algebra is real and 
spanned by $B_3,B_4,B_7,$ $B_8$. The stabilizer is the same as for $\langle X_{\alpha+\beta}, Y_{\alpha+\beta}, H_\alpha+H_\beta\rangle$, and in the same 
way we see that we get one algebra.  Denote 
\begin{equation}\w_{4}=\langle B_3,B_4,B_7, B_8\rangle.\end{equation}
\end{proof}

\begin{table}[H]
\renewcommand{\arraystretch}{1.6} 
\caption{The Real Subalgebras of $\mathfrak{su}(3)$ from Theorem \ref{su3}.} \label{susa}
\centering
\scalebox{0.98}{
\begin{tabular}{|c|c|c|}
\hline
Dim. & Subalgebra Representative & Conditions \\
\hline \hline
1 & $\w_1^\lambda=\langle B_7+\lambda B_8\rangle$, $\lambda \in \mathbb{R}$ &
$\w_1^\lambda \sim \w_1^\eta$ iff $\eta =\lambda, 1-\lambda, \frac{1}{\lambda}, \frac{\lambda-1}{\lambda}, \frac{\lambda}{\lambda-1},$ or  $\frac{1}{1-\lambda}$\\
\hline
2& $\w_2=\langle B_7, B_8 \rangle$& \\
\hline
3 & $\w_{3,1}=\langle B_3, B_4, B_7+B_8 \rangle$&\\
\hline
3 & $\w_{3,2}=\langle B_1+B_5, B_2+B_6, B_7+B_8 \rangle$ & \\
\hline 
4& $\w_4=\langle B_3, B_4, B_7, B_8 \rangle$& \\
\hline
\end{tabular} }
\end{table}

\appendix

\section*{Appendix: The complex subalgebras of $\mathfrak{sl}_3(\mathbb{C})$.}

We present the classification of complex subalgebras of $\mathfrak{sl}_3(\mathbb{C})$, up to conjugation in $\mathrm{SL}_3(\mathbb{C})$,  from \cite{dr16a, dr16b, dr18}.  The classification is contained in Tables \ref{table1sl3C} through \ref{table6sl3C}, organized according to dimension and structure.  Note  that for the tables in the appendix, $\u \sim \mathfrak{v}$   indicates that subalgebras  $\u$ and  $\mathfrak{v}$ are conjugate with respect to $\mathrm{SL}_3(\mathbb{C})$. The basis for $\mathfrak{sl}_3(\mathbb{C})$ (also a basis for $\mathfrak{sl}_3(\mathbb{R})$) used in the classification is 
$\{H_\alpha, H_\beta, X_\alpha, X_\beta,X_{\alpha+\beta}, Y_\alpha, Y_\beta,Y_{\alpha+\beta}\}$ 
defined in Eq. \eqref{sl3basiss}.

\begin{table} [H]\renewcommand{\arraystretch}{1.6} \caption{One-Dimensional Complex Subalgebras of $\mathfrak{sl}_3(\mathbb{C})$.   } \label{table1sl3C}\centering
\scalebox{0.97}{
\begin{tabular}{|c|c|c|c|c|c|}
\hline
 Subalgebra Representative \\
\hline \hline
 $\langle X_\alpha + X_\beta \rangle$ 
 \\
\hline
 $\langle X_\alpha  \rangle$  
 \\
\hline
 $\langle X_\alpha + H_\alpha+2H_\beta \rangle$  
 \\
\hline
 $\langle H_\alpha + a H_\beta \rangle$  
 \\
   $\langle H_\alpha + a H_\beta \rangle \sim \langle H_\alpha + b H_\beta \rangle$  iff $b=a, \frac{1}{a}, 1-a, \frac{1}{1-a}, \frac{a}{a-1}$, or $\frac{a-1}{a}$    
 \\
\hline
\end{tabular}}
\end{table}

\begin{table} [H]\renewcommand{\arraystretch}{1.6} \caption{Two-Dimensional Complex Subalgebras of $\mathfrak{sl}_3(\mathbb{C})$.   } \label{table2sl3C}\centering
\scalebox{0.97}{
\begin{tabular}{|c|c|c|c|c|c|c|}
\hline
 Subalgebra Representative\\
\hline \hline
 $\langle X_\alpha + X_\beta, X_{\alpha+\beta} \rangle$  
 \\
\hline
 $\langle X_\alpha, H_\alpha+2H_\beta  \rangle$  
 \\
\hline
 $\langle X_\alpha, X_{\alpha+\beta} \rangle$  
 \\
\hline
 $\langle X_\alpha,  Y_\beta \rangle$  
 \\
\hline
 $\langle H_\alpha,  H_\beta \rangle$   
 \\
\hline
 $\langle X_\alpha+X_\beta,  H_\alpha+H_\beta \rangle$   
 \\
\hline
 $\langle X_\alpha, -H_\alpha+H_\beta+3X_{\alpha+\beta}\rangle$   
 \\
\hline
 $\langle X_\alpha, -2H_\alpha-H_\beta+3Y_{\beta} \rangle$  
 \\
\hline
 $\langle X_\alpha, aH_\alpha+(2a+1)H_\beta \rangle$  
 \\
 $\langle X_\alpha, aH_\alpha+(2a+1)H_\beta \rangle \sim \langle X_\alpha, bH_\alpha+(2b+1)H_\beta \rangle$  
iff $a=b$
 \\
\hline
\end{tabular}}
\end{table}

\begin{table} [H]\renewcommand{\arraystretch}{1.6} \caption{Semisimple Complex Subalgebras of $\mathfrak{sl}_3(\mathbb{C})$.   } \label{table5sl3C}\centering
\scalebox{0.97}{
\begin{tabular}{|c|c|c|c|c|c||}
\hline
 Subalgebra Representative \\
\hline \hline
 $\langle X_{\alpha+\beta}, Y_{\alpha+\beta}, H_\alpha+H_\beta \rangle$  
 \\
\hline
 $\langle X_{\alpha}+X_\beta, 2Y_{\alpha}+2Y_{\beta}, 2H_\alpha+2H_\beta \rangle$  
 \\
\hline
\end{tabular}}
\end{table}

\begin{table} [H]\renewcommand{\arraystretch}{1.6} \caption{Three-Dimensional Solvable Complex Subalgebras of $\mathfrak{sl}_3(\mathbb{C})$.   } \label{table3sl3C}\centering
\scalebox{0.97}{
\begin{tabular}{|c|c|c|c|c|c|}
\hline
 Subalgebra Representative\\
\hline \hline
 $\langle X_\alpha, X_{\alpha+\beta}, 2H_\alpha+H_\beta \rangle$  
 \\
\hline
 $\langle X_\alpha, Y_\beta, H_\alpha-H_\beta  \rangle$  
 \\
\hline
 $\langle X_\alpha, X_{\alpha+\beta}, 2H_\alpha+H_\beta+X_\beta \rangle$  
 \\
\hline
 $\langle Y_\alpha, Y_{\alpha+\beta}, 2H_\alpha+H_\beta+X_\beta  \rangle$   
 \\
\hline
 $\langle X_\alpha+X_\beta, X_{\alpha+\beta}, H_\alpha+H_\beta \rangle$   
 \\
\hline
 $\langle X_\alpha,  H_\alpha, H_\beta \rangle$   
 \\
\hline
 $\langle X_\alpha, X_{\alpha+\beta}, (a-1)H_\alpha+aH_\beta\rangle, a \neq \pm 1$.   \\

$\langle X_\alpha, X_{\alpha+\beta}, (a-1)H_\alpha+aH_\beta\rangle \sim \langle X_\alpha, X_{\alpha+\beta}, (b-1)H_\alpha+b H_\beta\rangle$ iff  $a=b$ or $ab=1$ 
 \\
\hline
 $\langle X_\alpha, Y_\beta, H_\alpha+aH_\beta \rangle, a\neq \pm1$.  \\
 $\langle X_\alpha, Y_\beta, H_\alpha+aH_\beta \rangle \sim \langle X_\alpha, Y_\beta, H_\alpha+b H_\beta \rangle$ iff $a=b$ or $ab=1$ 
 \\
\hline
 $\langle X_\alpha, X_{\alpha+\beta}, H_\beta \rangle$  
 \\
\hline
 $\langle X_\alpha, Y_\beta, H_\alpha+H_\beta \rangle$  
 \\
\hline
 $\langle X_\alpha, X_{\beta}, X_{\alpha+\beta} \rangle$   
 \\
\hline
\end{tabular}}
\end{table}

\begin{table} [H]\renewcommand{\arraystretch}{1.6} \caption{Four- and Five-Dimensional Solvable Complex Subalgebras of $\mathfrak{sl}_3(\mathbb{C})$.   } \label{table4sl3C}\centering
\scalebox{0.97}{
\begin{tabular}{|c|c|c|c|c|c|}
\hline
Dim. & Subalgebra Representative \\
\hline \hline
4 & $\langle X_\alpha, X_{\alpha+\beta}, H_\alpha, H_\beta \rangle$  
 \\
\hline
4 & $\langle X_\alpha, Y_\beta, H_\alpha, H_\beta  \rangle$  
 \\
\hline
4 & $\langle X_\alpha, X_\beta, X_{\alpha+\beta}, H_\alpha+H_\beta \rangle$ 
 \\
\hline
4 & $\langle X_\alpha, X_\beta, X_{\alpha+\beta}, aH_\alpha+H_\beta  \rangle, a\neq \pm 1$. 
 \\
 & $\langle X_\alpha, X_\beta, X_{\alpha+\beta}, aH_\alpha+H_\beta  \rangle \sim \langle X_\alpha, X_\beta, X_{\alpha+\beta}, bH_\alpha+H_\beta  \rangle$  
iff $a=b$
 \\
\hline
4 & $\langle X_\alpha,  X_\beta, X_{\alpha+\beta}, H_\alpha \rangle$   
 \\
\hline
4 & $\langle X_\alpha, X_{\beta}, X_{\alpha+\beta}, H_\alpha-H_\beta\rangle$   
 \\
\hline
5 & $\langle X_\alpha, X_\beta, X_{\alpha+\beta}, H_\alpha, H_\beta \rangle$   
 \\
\hline
\end{tabular}}
\end{table}

\begin{table} [H]\renewcommand{\arraystretch}{1.6} \caption{Levi Decomposable Complex Subalgebras of $\mathfrak{sl}_3(\mathbb{C})$.   } \label{table6sl3C}\centering
\scalebox{0.97}{
\begin{tabular}{|c|c|c|c|c|c|}
\hline
Dim. & Subalgebra Representative \\
\hline \hline
4 & $\langle X_{\alpha+\beta}, Y_{\alpha+\beta}, H_\alpha+H_\beta \rangle \oplus \langle H_\alpha-H_\beta\rangle$  
 \\
\hline
5 & $\langle X_{\alpha+\beta}, Y_{\alpha+\beta}, H_\alpha+H_\beta \rangle \inplus \langle X_\alpha, Y_\beta \rangle$  
 \\
\hline
5 & $\langle X_{\alpha+\beta}, Y_{\alpha+\beta}, H_\alpha+H_\beta \rangle \inplus \langle X_\beta, Y_\alpha \rangle$  
 \\
\hline
6 & $\langle X_{\alpha+\beta}, Y_{\alpha+\beta}, H_\alpha+H_\beta \rangle \inplus \langle X_\alpha, Y_\beta, H_\alpha-H_\beta \rangle$  
 \\
\hline
6 & $\langle X_{\alpha+\beta}, Y_{\alpha+\beta}, H_\alpha+H_\beta \rangle \inplus \langle X_\beta, Y_\alpha, H_\alpha-H_\beta \rangle$  
 \\
\hline
\end{tabular}}
\end{table}

\begin{footnotesize}

\begin{table} [H]\renewcommand{\arraystretch}{1.6} \caption{Stabilizers of One-Dimensional Subalgebras of $\mathfrak{sl}(3,\mathbb{C})$ } \label{tab:stab1} 
\begin{tabular}{|c|c|c|clclclc|c|}
\hline
Subalgebra & Stabilizer\\
\hline \hline
$\langle X_{\alpha+\beta}\rangle$ & 
$\begin{pmatrix}
a & b & c\\
0 & d & e\\
0&0&f
\end{pmatrix}$, $adf=1$.
\\
\hline
$\langle X_{\alpha}\rangle$ & 
$\begin{pmatrix}
a & b & c\\
0 & d & 0\\
0&e&f
\end{pmatrix}$, $adf=1$.
\\
\hline
$\langle X_{\alpha}+H_\alpha+2H_\beta\rangle$ & 
$\begin{pmatrix}
a & b & 0\\
0 & a & 0\\
0&0&\frac{1}{a^2}
\end{pmatrix}$, $a\neq 0$.
\\
\hline
$\langle H_{\alpha}+\lambda H_\beta \rangle$ & 
$\begin{pmatrix}
a & b & 0\\
c & d & 0\\
0&0&e
\end{pmatrix}$, $e(ad-bc) =1$, if $\lambda=2$.
\\
& 
$\begin{pmatrix}
a & 0 & 0\\
0 & b & c\\
0&d&e
\end{pmatrix}$, $a(be-cd) =1$, if $\lambda=\frac{1}{2}$.
\\
& 
$\begin{pmatrix}
a & 0 & b\\
0 & c & 0\\
d&0&e
\end{pmatrix}$, $c(ae-bd) =1$, if $\lambda=-1$.
\\
& 
$\begin{pmatrix}
0& 0 & a\\
0 & b & 0\\
c&0&0
\end{pmatrix}$, $abc =-1$, or $\begin{pmatrix}
a& 0 & 0\\
0 & b & 0\\
0&0&c
\end{pmatrix}$, $abc =1$ if $\lambda=1$.
\\
& 
$\begin{pmatrix}
0& a & 0\\
b & 0 & 0\\
0&0&c
\end{pmatrix}$, $abc =-1$,  or $\begin{pmatrix}
a& 0 & 0\\
0 & b & 0\\
0&0&c
\end{pmatrix}$, $abc =1$ if $\lambda=0$.
\\
& 
$\begin{pmatrix}
a& 0 & 0\\
0 & b & 0\\
0&0&c
\end{pmatrix}$, $abc =1$, if $\lambda \neq  0, 1, -1, 2, \frac{1}{2}, \frac{1}{2}\pm \frac{i\sqrt{3}}{2}$.
\\
& 
$\begin{pmatrix}
a& 0 & 0\\
0 & b & 0\\
0&0&c
\end{pmatrix}$, $abc=1$, $\begin{pmatrix}
0& 0 & a\\
b & 0 & 0\\
0&c&0
\end{pmatrix}$, $abc =1$, or 
$\begin{pmatrix}
0& a & 0\\
0 & 0 & b\\
c&0&0
\end{pmatrix}$, $abc =1$, if $\lambda =\frac{1}{2}\pm \frac{i\sqrt{3}}{2}$.
\\
\hline
\end{tabular}
\end{table}

\end{footnotesize}

 \begin{tiny}

\begin{table} [H]\renewcommand{\arraystretch}{1.6} \caption{Stabilizers of Two-Dimensional Subalgebras of $\mathfrak{sl}(3,\mathbb{C})$ }\label{tab:stab2}
\begin{tabular}{|c|c|c|clclclc|c|}
\hline
Subalgebra & Stabilizer\\
\hline \hline
$\langle X_{\alpha}+X_{\beta}, X_{\alpha+\beta} \rangle$  & $\begin{pmatrix}
a^2b & c & d\\
0 & ab & e\\
0&0&b
\end{pmatrix}$, $a^3b^3=1$.
\\
\hline
$\langle X_{\alpha}, H_{\alpha}+2H_{\beta} \rangle$  & $\begin{pmatrix}
a & b & 0\\
0 & c & 0\\
0&0&d
\end{pmatrix}$, $acd=1$.
\\
\hline
$\langle H_{\alpha}, H_{\beta} \rangle$  & $\begin{pmatrix}
a & 0 & 0\\
0 & 0 & b\\
0&c&0
\end{pmatrix}$, $abc=-1$;   $\begin{pmatrix}
0 & a & 0\\
b & 0 & 0\\
0&0&c
\end{pmatrix}$, $abc=-1$;  $\begin{pmatrix}
0 & 0 & a\\
b & 0 & 0\\
0&c&0
\end{pmatrix}$, $abc=1$; \\
& 
$\begin{pmatrix}
0 & a & 0\\
0 & 0 & b\\
c&0&0
\end{pmatrix}$, $abc=1$;  $\begin{pmatrix}
a & 0 & 0\\
0 & b & 0\\
0&0&c
\end{pmatrix}$, $abc=1$;  $\begin{pmatrix}
0 & 0 & a\\
0 & b & 0\\
c&0&0
\end{pmatrix}$, $acb=-1$.
\\
\hline
$\langle X_{\alpha}+X_\beta,  H_{\alpha}+H_{\beta} \rangle$  & $\begin{pmatrix}
a^2b & -abc & \frac{bc^2}{2}\\
0 & ab& -bc\\
0&0&b
\end{pmatrix}$, $a^3b^3=1$.
\\
\hline
$\langle X_{\alpha},  Y_{\beta} \rangle$  & $\begin{pmatrix}
a & b & c\\
0 & d& 0\\
e&f&g
\end{pmatrix}$, $d(ag-ec)=1$.
\\
\hline
$\langle X_{\alpha},  -H_\alpha+H_\beta+3X_{\alpha+\beta}\rangle$  & $\begin{pmatrix}
a & b & c\\
0 & \frac{1}{a^2}& 0\\
0&0&a
\end{pmatrix}$, $a \neq 0$.
\\
\hline
$\langle X_{\alpha},  -2H_\alpha-H_\beta+3Y_{\beta}\rangle$  & $\begin{pmatrix}
\frac{1}{a^2} & b & 0\\
0 & a& 0\\
0&c&a
\end{pmatrix}$, $a \neq 0$.
\\
\hline
$\langle X_{\alpha},  X_{\alpha+\beta}\rangle$  & $\begin{pmatrix}
a & b & c\\
0 & d& e\\
0&f&g
\end{pmatrix}$, $a (dg-ef)=1$.
\\
\hline
$\langle X_{\alpha},  \lambda H_\alpha+(2\lambda+1)H_\beta \rangle$  & $\begin{pmatrix}
a & b & 0\\
0 & c& 0\\
0&0&d
\end{pmatrix}$, $acd=1$, if $\lambda \neq -\frac{2}{3}, -\frac{1}{3}$.
  \\ 
 & $\begin{pmatrix}
a & b & 0\\
0 & c& 0\\
0&e&d
\end{pmatrix}$, $acd=1$, if $\lambda= -\frac{2}{3}$.\\
   & $\begin{pmatrix}
a & b & e\\
0 & c& 0\\
0&0&d
\end{pmatrix}$, $acd=1$, if $\lambda= -\frac{1}{3}$.
 \\
\hline
\end{tabular}
\end{table}

\end{tiny}

 \begin{tiny}

\begin{table} [H]\renewcommand{\arraystretch}{1.6} \caption{Stabilizers of Three-Dimensional Solvable Subalgebras of $\mathfrak{sl}(3,\mathbb{C})$ } \label{tab:stab3}
\begin{tabular}{|c|c|c|clclclc|c|} 
\hline
Subalgebra & Stabilizer\\
\hline \hline
$ \langle X_\alpha, X_{\alpha+\beta}, (\lambda-1)H_{\alpha}+\lambda H_\beta \rangle$, $\lambda \neq \pm 1$ & 
$\begin{pmatrix}
a & b & c\\
0 & d& 0 \\
0&0&e
\end{pmatrix}$,  $ade=1$.
\\
\hline
$ \langle X_\alpha, X_{\alpha+\beta}, 2H_{\alpha}+H_\beta \rangle$ & 
$\begin{pmatrix}
a & b & c\\
0 & d& e \\
0&f&g
\end{pmatrix}$,  $a(dg-ef)=1$.
\\
\hline
$ \langle X_\alpha, Y_\beta, H_{\alpha}+\lambda H_\beta \rangle$, $\lambda \neq \pm1$ & 
$\begin{pmatrix}
a & b & 0\\
0 & c& 0 \\
0&d&e
\end{pmatrix}$,  $ace=1$. \\
\hline
$ \langle X_\alpha, X_\beta, X_{\alpha+\beta} \rangle$ & 
$\begin{pmatrix}
a & b & c\\
0 & d & e \\
0&0&f
\end{pmatrix}$,  $adf=1$. \\
\hline
$ \langle X_\alpha, Y_{\beta}, H_\alpha+H_\beta \rangle$ & $\begin{pmatrix}
a & b & 0\\
0 & c & 0 \\
0&d&e
\end{pmatrix}$,  $ace=1$, or 
 $\begin{pmatrix}
0 & a & b\\
0 & c & 0 \\
d&e&0
\end{pmatrix}$,  $-bcd=1$
 \\
\hline
$ \langle X_\alpha+X_\beta, X_{\alpha+\beta}, H_\alpha+H_\beta \rangle$ & $\begin{pmatrix}
a^2b & -abc & d\\
0 & ab & -bc \\
0&0&b
\end{pmatrix}$,  $a^3b^3=1$. \\
\hline
$ \langle X_\alpha, X_{\alpha+\beta}, 2H_\alpha+H_\beta+X_\beta \rangle$ & $\begin{pmatrix}
\frac{1}{a^2} & b & c\\
0 & a & d\\
0&0&a
\end{pmatrix}$, $a\neq 0$. \\
\hline
$ \langle Y_\alpha, Y_{\alpha+\beta}, 2H_\alpha+H_\beta+X_\beta \rangle$ & $\begin{pmatrix}
\frac{1}{a^2} & 0 & 0\\
b & a & c\\
d&0&a
\end{pmatrix}$, $a\neq 0$. \\
\hline
$ \langle X_\alpha, Y_{\beta}, H_\alpha-H_\beta \rangle$ & $\begin{pmatrix}
a & b & c\\
0 & d & 0\\
e&f&g
\end{pmatrix}$, $d(ag-ce)=1$. \\
\hline
$\langle X_\alpha, H_\alpha, H_\beta \rangle$ & $\begin{pmatrix}
a & b & 0\\
0 & c & 0\\
0&0&d
\end{pmatrix}$, $acd=1$.\\
\hline
$\langle X_\alpha, X_{\alpha+\beta}, H_\beta \rangle$ & 
$\begin{pmatrix}
a & b & c\\
0 & d & 0\\
0&0&e
\end{pmatrix}$, $ade =1$; or $\begin{pmatrix}
a & b & c\\
0 & 0 & d\\
0&e&0
\end{pmatrix}$, $ade=-1$.\\
\hline
\end{tabular}
\end{table}

\end{tiny}

\begin{tiny}

\begin{table} [H]\renewcommand{\arraystretch}{1.6} \caption{Stabilizers of Three-Dimensional Semisimple Subalgebras of $\mathfrak{sl}(3,\mathbb{C})$ } \label{tab:stab4}
\begin{tabular}{|c|c|c|clclclc|c|} 
\hline
Subalgebra & Stabilizer\\
\hline \hline
$\langle X_\alpha+X_\beta, 2Y_\alpha+2Y_\beta, 2H_\alpha+2H_\beta\rangle$ & 
$\mathrm{PSL}(2,\C)\times \mu_3$ where \\
& the Lie algebra of $\mathrm{PSL}(2,\C)$ is exactly the subalgebra considered here\\
& and $\mu_3$ is the subgroup consisting of $\diag(\omega,\omega,\omega)$ with $\omega^3=1$
\\
\hline
$\langle X_{\alpha+\beta}, Y_{\alpha+\beta}, H_\alpha+H_\beta \rangle$ & 
$\begin{pmatrix}
a & 0 & b\\
0 & c & 0\\
d&0&e
\end{pmatrix}$, $c(ae-bd)=1$.\\
\hline
\end{tabular}
\end{table}

\end{tiny}

\begin{small}

\begin{table} [H]\renewcommand{\arraystretch}{1.6} \caption{Stabilizers of Four-Dimensional Subalgebras of $\mathfrak{sl}(3,\mathbb{C})$ } \label{tab:stab5}
\begin{tabular}{|c|c|c|clclclc|c|} 
\hline
Subalgebra & Stabilizer\\
\hline \hline
$\langle X_\alpha, X_{\alpha+\beta}, H_\alpha, H_\beta \rangle$ & $\begin{pmatrix}
a & b & c\\
0 & d & 0\\
0&0&e
\end{pmatrix}$, $ade=1$, or
 $\begin{pmatrix}
a & b & c\\
0 & 0 & d\\
0&e&0
\end{pmatrix}$, $ade=-1$.
 \\
\hline
$\langle X_\alpha, X_\beta, X_{\alpha+\beta}, H_\alpha \rangle$ & $\begin{pmatrix}
a & b & c\\
0 & d & e\\
0&0&f
\end{pmatrix}$, $adf=1$.
 \\
\hline
$\langle X_\alpha, X_\beta, X_{\alpha+\beta}, H_\alpha-H_\beta \rangle$ & $\begin{pmatrix}
a & b & c\\
0 & d & e\\
0&0&f
\end{pmatrix}$, $adf=1$.
 \\
\hline
$\langle X_\alpha, X_\beta, X_{\alpha+\beta}, H_\alpha+H_\beta \rangle$ & 
$\begin{pmatrix}
a & b & c\\
0 & d & e\\
0&0&f
\end{pmatrix}$, $adf=1$.
 \\
\hline
$\langle X_\alpha, X_\beta, X_{\alpha+\beta}, H_\alpha \rangle$ & $\begin{pmatrix}
a & b & c\\
0 & d & e\\
0&0&f
\end{pmatrix}$, $adf=1$.
 \\
\hline
$\langle X_\beta, Y_\alpha, Y_{\alpha+\beta}, 2H_\alpha+H_\beta \rangle$ & $\begin{pmatrix}
a & 0 & 0\\
b & c & d\\
e&0&f
\end{pmatrix}$, $acf=1$.
 \\
\hline
$\langle X_\alpha, Y_\beta, H_{\alpha}, H_\beta \rangle$ & $\begin{pmatrix}
a & b & 0\\
0 & c & 0\\
0&d&e
\end{pmatrix}$, $ace=1$; or 
 $\begin{pmatrix}
0 & a & b\\
0 & c & 0\\
d&e&0
\end{pmatrix}$, $dbc=-1$.
 \\
\hline
$\langle X_\alpha, X_\beta, X_{\alpha+\beta}, \lambda H_\alpha+H_\beta \rangle$, $\lambda \neq \pm 1$ & 
 $\begin{pmatrix}
a & b & c\\
0 & d & e\\
0&0&f
\end{pmatrix}$, $abe=1$.
 \\
\hline
$\langle X_{\alpha+\beta}, Y_{\alpha+\beta}, H_\alpha+H_\beta\rangle \oplus \langle  H_\alpha-H_\beta \rangle$ & 
 $\begin{pmatrix}
a & 0 & b\\
0 & c & 0\\
d&0&e
\end{pmatrix}$, $c(ae-bd)=1$.
\\
\hline 
\end{tabular}
\end{table}

\end{small}

\begin{table} [H]\renewcommand{\arraystretch}{1.6} \caption{Stabilizers of Five-Dimensional Subalgebras of $\mathfrak{sl}(3,\mathbb{C})$ }\label{tab:stab6}
\begin{tabular}{|c|c|c|clclclc|c|} 
\hline
Subalgebra & Stabilizer\\
\hline \hline
$\langle X_\alpha, X_\beta, X_{\alpha+\beta}, H_\alpha, H_\beta \rangle$ & $\begin{pmatrix}
a & b & c\\
0 & d & e\\
0&0&f
\end{pmatrix}$, $adf=1$. \\ 
\hline
$\langle X_{\alpha+\beta}, Y_{\alpha+\beta}, H_\alpha+H_\beta\rangle \inplus \langle X_\alpha, Y_\beta\rangle$ & 
$\begin{pmatrix}
a & b & c\\
0 & d & 0\\
e&f&g
\end{pmatrix}$, $d(ag-ce)=1$.
\\
\hline
$\langle X_{\alpha+\beta}, Y_{\alpha+\beta}, H_\alpha+H_\beta\rangle \inplus \langle X_\beta, Y_\alpha\rangle$ & 
$\begin{pmatrix}
a & 0 & b\\
c & d & e\\
f&0&g
\end{pmatrix}$, $d(ag-bf)=1$.
\\
\hline
\end{tabular}
\end{table}

\begin{table} [H]\renewcommand{\arraystretch}{1.6} \caption{Stabilizers of\textit{ } Six-Dimensional Subalgebras of $\mathfrak{sl}(3,\mathbb{C})$ } \label{tab:stab7}
\begin{tabular}{|c|c|c|clclclc|c|}
\hline
Subalgebra & Stabilizer\\
\hline \hline
$\langle X_{\alpha+\beta}, Y_{\alpha+\beta}, H_\alpha+H_\beta\rangle \inplus \langle X_\alpha, Y_\beta, H_\alpha-H_\beta\rangle$ & 
$\begin{pmatrix}
a & b & c\\
0 & d & 0\\
e&f&g
\end{pmatrix}$, $d(ag-ce)=1$.
\\
\hline
$\langle X_{\alpha+\beta}, Y_{\alpha+\beta}, H_\alpha+H_\beta\rangle \inplus \langle X_\beta, Y_\alpha, H_\alpha-H_\beta\rangle$ &
$\begin{pmatrix}
a & 0 & b\\
c & d & e\\
f&0&g
\end{pmatrix}$, $d(ag-bf)=1$.
 \\
\hline
\end{tabular}
\end{table}


\begin{thebibliography}{99}

\bibitem{borovoi20}
M.~Borovoi,
Real points in a homogeneous space of a real algebraic group,
\textit{arXiv preprint} \texttt{arXiv:2106.14871 [math.AG]}, 2020.

\bibitem{bg}
M.~Borovoi and W.~A.~de~Graaf,
Computing Galois cohomology of a real linear algebraic group,
\textit{J. Lond. Math. Soc.} (2) \textbf{109} (2024), no.~5, Paper No.~e12906, 53~pp.

\bibitem{bgl}
M.~Borovoi, W.~A.~de~Graaf, and H.~V.~L\^e,
Real graded Lie algebras, Galois cohomology, and classification of trivectors in~$\mathbb{R}^9$,
\textit{arXiv preprint} \texttt{arXiv:2106.00246 [math.RT]}, 2021.

\bibitem{boson}
F.~Iachello and A.~Arima,
\textit{The Interacting Boson Model},
Cambridge Monogr. Math. Phys., Cambridge Univ. Press, 1987.

\bibitem{bou9}
N.~Bourbaki,
\textit{\'El\'ements de math\'ematique. Groupes et alg\`ebres de Lie. Chapitre~9},
Masson, Paris, 1982.

\bibitem{ddg}
A.~Douglas and W.~A.~de~Graaf,
Galois cohomology and the subalgebras of the special linear algebra~$\mathfrak{sl}_3(\mathbb{R})$,
\textit{arXiv preprint} \texttt{arXiv:2502.00810 [math.RA]}, 2024.

\bibitem{ddg2}
A.~Douglas and W.~A.~de~Graaf,
The subalgebras of the generalized special unitary algebra~$\mathfrak{su}(2,1)$,
\textit{arXiv preprint} \texttt{arXiv:2504.17942 [math.RA]}, 2025.

\bibitem{dr16a}
A.~Douglas and J.~Repka,
The subalgebras of~$A_2$,
\textit{arXiv preprint} \texttt{arXiv:1509.00932 [math.RA]}, 2015, updated 2024.

\bibitem{dr16b}
A.~Douglas and J.~Repka,
A classification of the subalgebras of~$A_2$,
\textit{J. Pure Appl. Algebra} \textbf{220} (2016), no.~6, 2389--2413.

\bibitem{dr18}
A.~Douglas and J.~Repka,
Subalgebras of the rank two semisimple Lie algebras,
\textit{Linear Multilinear Algebra} \textbf{66} (2018), no.~10, 2049--2075.

\bibitem{dr25}
A.~Douglas and J.~Repka,
Narrow and wide regular subalgebras of semisimple Lie algebras,
\textit{J. Algebra} \textbf{664} (2025), 348--361.

\bibitem{gowa}
R.~Goodman and N.~R.~Wallach,
\textit{Symmetry, Representations, and Invariants},
Grad. Texts in Math., vol.~255, Springer, Dordrecht, 2009.

\bibitem{haydon}
P.~E.~Haydon,
\textit{Symmetry Methods for Differential Equations},
Cambridge Univ. Press, 2000.

\bibitem{janisse}
T.~Janisse,
\textit{The Real Subalgebras of $\mathfrak{so}_4(\mathbb{C})$ and $G_{2(2)}$},
Doctoral dissertation, Univ. of Toronto, 2023.

\bibitem{olver}
P.~J.~Olver,
\textit{Applications of Lie Groups to Differential Equations},
Springer-Verlag, 1986.

\bibitem{pw77}
J.~Patera and P.~Winternitz,
Subalgebras of real three- and four-dimensional Lie algebras,
\textit{J. Math. Phys.} \textbf{18} (1977), no.~7, 1449--1455.

\bibitem{pwz}
J.~Patera, P.~Winternitz, and H.~Zassenhaus,
Maximal Abelian subalgebras of real and complex symplectic Lie algebras,
\textit{J. Math. Phys.} \textbf{24} (1983), no.~8, 1973--1985.

\bibitem{sansuc}
J.-J.~Sansuc,
Groupe de Brauer et arithm\'etique des groupes alg\'ebriques lin\'eaires sur un corps de nombres,
\textit{J. Reine Angew. Math.} \textbf{327} (1981), 12--80.

\bibitem{serre}
J.-P.~Serre,
\textit{Galois Cohomology},
Springer-Verlag, Berlin, 1997.

\bibitem{sophus}
S.~Lie,
\textit{Klassifikation und Integration von Gew\"ohnlichen Differentialgleichungen zwischen $x$, $y$, die eine Gruppe von Transformationen Gestatten},
\textit{Gesammelte Abhandlungen}, vol.~5, Teubner, Leipzig, 1924.

\bibitem{sug}
M.~Sugiura,
Conjugate classes of Cartan subalgebras in real semi-simple Lie algebras,
\textit{J. Math. Soc. Jpn.} \textbf{11} (1959), 374--434.

\bibitem{win04}
P.~Winternitz,
Subalgebras of Lie algebras: Example of~$\mathfrak{sl}(3,\mathbb{R})$,
in: \textit{Symmetry in Physics}, CRM Proc. Lecture Notes, vol.~34, Amer. Math. Soc., 2004, pp.~215--227.

\end{thebibliography}
\end{document}